\documentclass[a4paper,oneside,11pt]{article}
 
\usepackage[a4paper]{geometry}
\usepackage{aeguill}                 
\usepackage{graphicx}
\usepackage{amsmath,amsfonts,amssymb,amsthm}
\usepackage{pstricks} 
\usepackage{epstopdf}
\usepackage{srcltx}
\usepackage{comment}
\usepackage[utf8]{inputenc} 
\usepackage[T1]{fontenc}

\newtheorem{definition}{Definition} [section]
\newtheorem{de}[definition]{Definition}
 
\newtheorem{theo}[definition]{Theorem}
\newtheorem{proposition}[definition]{Proposition}
\newtheorem{prop}[definition]{Proposition}
\newtheorem{lemma}[definition]{Lemma}
\newtheorem{corollary}[definition]{Corollary}
\newtheorem{cor}[definition]{Corollary}

\newtheorem{remark}[definition]{Remark}

\newtheorem*{oquestion}{Open question}
\newtheorem{theointro}{Theorem}
\newtheorem*{defintro}{Definition}
\newtheorem{corintro}[theointro]{Corollary}
\newtheorem{ex}[definition]{Example}

\def\aut{{\rm{Aut}}}
\def\out{{\rm{Out}}}
\def\G{\Gamma}

\def\stab{{\rm{Stab}}}

\newcommand{\Out}{\mathrm{Out}}
\newcommand{\Comm}{\mathrm{Comm}}
\newcommand{\Stab}{\mathrm{Stab}}
\newcommand{\calf}{\mathcal{F}}

\newcommand{\rk}{\mathrm{rk}}
\newcommand{\pro}{\mathrm{prod}}

\newcommand{\Aut}{\mathrm{Aut}}
\newcommand{\ad}{\mathrm{ad}}
\newcommand{\calt}{\mathcal{T}}

\newcommand{\calz}{\mathcal{Z}}
\newcommand{\Mod}{\text{Mod}}
\newcommand{\zmax}{\calz_{\max}}
\newcommand{\ia}{\mathrm{IA}_N(\mathbb{Z}/3\mathbb{Z})}

\newcommand{\calc}{\mathcal{C}}
\newcommand{\dunion}{\sqcup}

\newcommand{\Fix}{\mathrm{Fix}}
\newcommand{\ens}{\mathrm{FS}^{\mathrm{ens}}}
\newcommand{\id}{\mathrm{id}}
\newcommand{\calp}{\mathcal{P}}
\newcommand{\ns}{\mathrm{FS}^{\mathrm{ns}}}
\newcommand{\calq}{\mathcal{Q}}
\newcommand{\FF}{\mathrm{FF}}
\newcommand{\Wh}{\mathrm{Wh}}
\newcommand{\calb}{\mathcal{B}}
\newcommand{\GL}{\mathrm{GL}}

\newcommand{\iat}{\mathrm{IA}_3(\mathbb{Z}/3\mathbb{Z})}

\newcommand{\cala}{\mathcal{A}}

\title{Rigidity of the Torelli subgroup in $\Out(F_N)$}
\author{Sebastian Hensel, Camille Horbez and Richard D. Wade}
\begin{document}

\maketitle

\begin{abstract}
Let $N\ge 4$. We prove that every injective homomorphism from the Torelli subgroup $\mathrm{IA}_N$ to $\mathrm{Out}(F_N)$ differs from the inclusion by a conjugation in $\mathrm{Out}(F_N)$. This applies more generally to the following subgroups of $\Out(F_N)$: every finite-index subgroup of $\mathrm{Out}(F_N)$ (recovering a theorem of Farb and Handel); every subgroup of $\mathrm{Out}(F_N)$ that contains a finite-index subgroup of one of the groups in the Andreadakis--Johnson filtration of $\mathrm{Out}(F_N)$; every subgroup that contains a power of every linearly-growing automorphism; more generally, every \emph{twist-rich} subgroup of $\Out(F_N)$ -- those are subgroups that contain sufficiently many twists in an appropriate sense.

Among applications, this recovers the fact that the abstract commensurator of every group above is equal to its relative commensurator in $\mathrm{Out}(F_N)$; it also implies that all subgroups in the Andreadakis--Johnson filtration of $\Out(F_N)$ are co-Hopfian.

We also prove the same rigidity statement for subgroups of $\Out(F_3)$ which contain a power of every Nielsen transformation. This shows in particular that $\Out(F_3)$ and all its finite-index subgroups are co-Hopfian, extending a theorem of Farb and Handel to the $N=3$ case.
\end{abstract}

\section*{Introduction}

Let us start with the following question, that we learned from Farb, which to our knowledge is still open.

\begin{oquestion}
Let $g\ge 2$, and let $\Sigma_g$ be a closed oriented surface of genus $g$. Let $\Mod(\Sigma_g)$ be the mapping class group of $\Sigma_g$, and $\calt_g\subseteq\Mod(\Sigma_g)$ be its Torelli subgroup, i.e.\ the kernel of the action on homology with $\mathbb{Z}$ coefficients.  

Is every injective homomorphism $\calt_g\to\Mod(\Sigma_g)$ induced from conjugation by an element in $\Mod^{\pm}(\Sigma_g)$? Or do there exist `exotic' embeddings of the Torelli subgroup?  
\end{oquestion}

While the question for mapping class groups is still open to our knowledge, the goal of the present paper is to answer the analogous question for the outer automorphism group $\Out(F_N)$ of a finitely generated free group $F_N$ with $N\ge 4$, by showing that there are no exotic embeddings of the kernel $\mathrm{IA}_N$ of the natural map $\Out(F_N)\to\GL_N(\mathbb{Z})$ into $\Out(F_N)$. Actually, as will be explained below, our result applies to a much wider collection of subgroups of $\Out(F_N)$ than just $\mathrm{IA}_N$.

This is initially quite surprising, as the study of $\Out(F_N)$ is traditionally harder than that of mapping class groups. However,  the subgroup  structure of $\Out(F_N)$ is more intricate, and this can be leveraged to our advantage for rigidity questions. In particular, following a strategy that was suggested by Martin Bridson and initiated by the last two named authors in their previous work \cite{HW} on commensurations of subgroups of $\Out(F_N)$, we will take advantage of certain direct products of free groups in $\Out(F_N)$ that arise as groups of twists associated to free splittings of $F_N$. These subgroups do not have a natural analogue in mapping class groups. 

The study of algebraic rigidity in $\Out(F_N)$ started with works of Khramtsov \cite{Khr} and Bridson--Vogtmann \cite{BV2} who proved that for every $N\ge 3$, every automorphism of $\Out(F_N)$ is inner. Their work is closely related to a more geometric rigidity statement regarding the symmetries of Outer space: Bridson--Vogtmann showed in \cite{BV} that the group of simplicial automorphisms of the spine of reduced Outer space is precisely $\Out(F_N)$. Later, Farb--Handel proved in \cite{FH} that for every $N\ge 4$, the group $\Out(F_N)$ is equal to its own abstract commensurator: this means that every isomorphism between two finite index subgroups of $\Out(F_N)$ is given by conjugation by an element of $\Out(F_N)$. Meanwhile, many other simplicial complexes equipped with an $\Out(F_N)$-action and closely related to Outer space were proven to be rigid, i.e.\ to have $\Out(F_N)$ as their automorphism group \cite{AS,Pan,HW2,BB,HW}. While Farb and Handel's proof of the rigidity of commensurations was algebraic, a new proof, which is more geometric and relies on those various complexes, was recently given by the last two named authors in \cite{HW}: this new proof extended the Farb--Handel theorem to the $N=3$ case, and also computed the abstract commensurator of many interesting subgroups, including $\mathrm{IA}_N$ and, when $N \geq 4$, all subgroups from the Andreadakis--Johnson filtration (the $k^{\text{th}}$ term of the Andreadakis--Johnson filtration is the kernel of the natural map from $\Out(F_N)$ to the outer automorphism group of a free nilpotent group of order $N$ and of class $k$). 

In their work, Farb and Handel did not only study commensurations: they showed that (for $N\ge 4$) every injective homomorphism from a finite-index subgroup $\Gamma\subseteq\Out(F_N)$ to $\Out(F_N)$ differs from the inclusion by an inner automorphism of $\Out(F_N)$. As a consequence, every finite-index subgroup $\Gamma\subseteq\Out(F_N)$ is \emph{co-Hopfian}, i.e.\ every injective homomorphism from $\Gamma$ to itself is in fact an automorphism. 

The present paper aims at extending this to some infinite-index subgroups of $\Out(F_N)$. Our methods apply to all \emph{twist-rich} subgroups of $\Out(F_N)$. These include $\mathrm{IA}_N$, all subgroups from the Andreadakis--Johnson filtration of $\Out(F_N)$, and also all subgroups of $\Out(F_N)$ that contain a power of every linearly-growing automorphism (an example of such a subgroup is the kernel of the natural map to the outer automorphism group of a free Burnside group of rank $N$ and any exponent). Twist-rich subgroups satisfy the following stability properties: finite-index subgroups of twist-rich subgroups are twist-rich, and every subgroup of $\Out(F_N)$ that contains a twist-rich subgroup is also twist-rich. The definition is the following (see Section~\ref{sec:twists-background} for the definition of twists associated to a splitting).

\begin{defintro}[Twist-rich subgroups]
Let $N\ge 4$. A subgroup $\Gamma\subseteq\ia$ is \emph{twist-rich} if for every free splitting $S$ of $F_N$ such that $S/F_N$ is a rose with nonabelian vertex stabilizer $G_v$, and every half-edge $e$ adjacent to $v$, the intersection of $\Gamma$ with the group of twists about $e$ is nonabelian and viewed as a subgroup of $G_v$, it is not contained in any proper free factor of $G_v$. 
\end{defintro}

The main theorem of the present paper is the following.

\begin{theointro}\label{theo:intro-main}
Let $N\ge 4$, and let $\Gamma$ be a twist-rich subgroup of $\Out(F_N)$. Then every injective homomorphism $\Gamma\to\Out(F_N)$ differs from the inclusion by an inner automorphism of $\Out(F_N)$.
\end{theointro}

Two subgroups $H,H'\subseteq\Out(F_N)$ are \emph{commensurable} in $\Out(F_N)$ if the intersection $H\cap H'$ is finite-index in both $H$ and $H'$. A first consequence of our main theorem is the following.

\begin{corintro}\label{cor-co-hopfian}
Let $N\ge 4$, and let $H\subseteq\Out(F_N)$ be a subgroup which is commensurable in $\Out(F_N)$ to a normal twist-rich subgroup of $\Out(F_N)$. Then $H$ is co-Hopfian.
\end{corintro}

We denote by $N_{\Out(F_N)}(\Gamma)$ the normalizer of $\Gamma$ in $\Out(F_N)$, and by $\Comm_{\Out(F_N)}(\Gamma)$ its \emph{relative commensurator} in $\Gamma$, defined as the subgroup made of all elements $\Phi\in\Out(F_N)$ such that $\Gamma\cap\Phi\Gamma\Phi^{-1}$ has finite index in both $\Gamma$ and $\Phi\Gamma\Phi^{-1}$. The following consequence of Theorem~\ref{theo:intro-main} was essentially established in \cite{HW}, although we would like to stress that our definition of twist-rich subgroups in the present paper is \emph{a priori} slightly more general (and less technical) than the one from \cite{HW}.

\begin{corintro}
Let $N\ge 4$, and let $\Gamma$ be a twist-rich subgroup of $\Out(F_N)$. Then
\begin{enumerate}
\item The natural map $N_{\Out(F_N)}(\Gamma)\to\Aut(\Gamma)$ is an isomorphism.
\item The natural map $\Comm_{\Out(F_N)}(\Gamma)\to\Comm(\Gamma)$ is an isomorphism.
\end{enumerate}
\end{corintro}

We also prove a version of our main theorem in rank $3$.

\begin{theointro}\label{theo:intro-rank-3}
Let $\Gamma\subseteq\Out(F_3)$ be a subgroup that contains a power of every Nielsen automorphism. Then every injective homomorphism $\Gamma\to\Out(F_3)$ differs from the inclusion by an inner automorphism of $\Out(F_3)$.  
\end{theointro}

\begin{corintro}\label{cor:intro-co-hopf-3}
Let $\Gamma\subseteq\Out(F_3)$ be a subgroup which is commensurable in $\Out(F_N)$ to a normal subgroup that contains a power of every Nielsen automorphism. Then $\Gamma$ is co-Hopfian.
\end{corintro}

In particular $\Out(F_3)$, and more generally every finite-index subgroup of $\Out(F_3)$, is co-Hopfian. This result is already new to our knowledge.

\paragraph*{A word about the proof.} 
The broad outline of the paper follows that of \cite{HW} on commensurator rigidity: to deduce that an injection $f:\G \to G$ is induced by conjugation we want to combine \emph{combinatorial rigidity} of a graph  equipped with a $G$-action with \emph{algebraic characterizations} of the vertex stabilizers of this graph in $G$.  The blueprint in this general context is described in Section~\ref{sec:blueprint}. Typically, when passing from commensurator rigidity to the study of arbitrary injective homomorphisms, one needs the following:

\begin{itemize}
 \item The rigidity of the graph needs to be improved to show that arbitrary \emph{injective graph maps} are actually automorphisms.
 \item One loses control over centralizers and normalizers of subgroups under $f$: algebraic characterizations of vertex stabilizers need to be independent of these notions.
\end{itemize}

The proof splits into two main parts to tackle the above. As in \cite{HW}, the graph we work with is the \emph{edgewise nonseparating free splitting graph} $\ens$, whose vertices are one-edge nonseparating free splittings of $F_N$, where two splittings are joined by an edge whenever they have a common refinement whose quotient graph is a two-petal rose. 

First, we show that every injective graph map of $\ens$ is an automorphism. The automorphism group of this graph was shown in \cite{HW} to coincide with $\Out(F_N)$, building on the pioneering work of Bridson and Vogtmann \cite{BV}.  To show that injective maps are also automorphisms requires a delicate analysis of links in $\ens$. This is carried out in Section~\ref{sec:complexes}. Much of the intuition for this analysis comes from viewing splittings as sphere systems. This is combined with the combinatorial language of Culler and Vogtmann's work \cite{CV}, which we found to give clean and precise proofs to back up this intuition (and pictures). 

Second, we show that every injective homomorphism $f:\Gamma\to\Out(F_N)$ sends the $\Gamma$-stabilizer of a one-edge nonseparating free splitting $S$ into the $\Out(F_N)$-stabilizer of a unique one-edge nonseparating free splitting $S'$. In fact the assignment $S\mapsto S'$ determines an injective graph map of $\ens$, and this enables us to conclude by combining the general framework with rigidity of $\ens$. A key idea in our proof of this second point is to observe that, when $\Gamma\subseteq\ia$, the $\Gamma$-stabilizer of $S$ contains a normal subgroup -- contained in the group of twists about the splitting -- isomorphic to a direct product of two nonabelian free groups. On the other hand, building on \cite{HW,GH}, we get that conversely, every subgroup $H\subseteq\Out(F_N)$ which contains a direct product of two nonabelian free groups that are both normal in $H$, must preserve a nontrivial splitting $U$ of $F_N$ (which might have nontrivial edge stabilizers, but we get some control on its structure). This relies on the idea that direct products act badly on hyperbolic spaces (while many subgroups of $\Out(F_N)$ have nonelementary actions on hyperbolic complexes), together with a construction from \cite{GH} which canonically associates a nontrivial splitting of $F_N$ to every collection of free splittings with infinite elementwise stabilizer. We finally need a few extra arguments to exclude the possibility that $U$ could be a non-free splitting: these include taking advantage of maximal direct products of $2N-4$ nonabelian free groups in $\Out(F_N)$, that were extensively studied in \cite{BW}.

\paragraph*{Structure of the paper.} The paper is organized in the following way. Section~\ref{sec:background} contains background material about $\Out(F_N)$. In Section~\ref{sec:blueprint}, we set up a general framework for using rigidity of graphs to establish rigidity of embeddings. In Section~\ref{sec:complexes}, we show that when $N\ge 4$, every injective graph map of $\ens$ is a graph automorphism (with a slight adaptation in rank $3$). In Section~\ref{sec:commuting-normal}, we show that subgroups of $\Out(F_N)$ that contain normal direct products stabilize splittings. In Section~\ref{sec:twist-rich}, we introduce twist-rich subgroups and discuss some of their properties. In Section~\ref{sec:algebra}, we show that when $N\ge 4$, every injective homomorphism $\Gamma\to\Out(F_N)$ induces an injective graph map of $\ens$ and conclude our proof of Theorem~\ref{theo:intro-main}. Finally, the case of subgroups of $\Out(F_3)$ (Theorem~\ref{theo:intro-rank-3}) is tackled separately in Section~\ref{sec:rank-3}. 

\paragraph*{Acknowledgments.} We thank Kai-Uwe Bux for pointing us the importance of the normality assumption in Corollary~\ref{cor-co-hopfian}.

The second named author acknowledges support from the Agence Nationale de la Recherche under Grant ANR-16-CE40-0006.

\setcounter{tocdepth}{1}
\tableofcontents

\section{Background}\label{sec:background}

\emph{In this section we introduce relevant background for the paper. This covers: the finite-index subgroup $\ia$ and its role in avoiding periodic behaviour; relative automorphism groups and actions on free factor graphs; stabilizers of splittings and their associated groups of twists; and the behaviour of maximal direct products of free groups in $\out(F_N)$.}

\subsection{Properties of $\ia$}

Let $N\ge 3$. We recall that $\ia$ is the kernel of the natural map from $\Out(F_N)$ to $\mathrm{GL}(N,\mathbb{Z}/3\mathbb{Z})$ given by the action on homology with coefficients in $\mathbb{Z}/3\mathbb{Z}$. Outer automorphisms of $F_N$ that belong to $\ia$ satisfy several nice properties that will enable us to avoid dealing with some finite-order phenomena.

\begin{theo}[{Handel--Mosher \cite[Theorem~3.1]{HM2}}]\label{theo:hm}
Let $H\subseteq\ia$ be a subgroup, and let $A\subseteq F_N$ be a free factor whose conjugacy class is invariant by some finite-index subgroup of $H$. Then the conjugacy class of $A$ is $H$-invariant.
\end{theo}

We recall that a \emph{splitting} of $F_N$ is a minimal, simplicial $F_N$-action on a simplicial tree. It is a \emph{free splitting} if all edge stabilizers are trivial. A consequence of the above theorem of Handel and Mosher is the following, see \cite[Lemma~2.6]{HW}.

\begin{theo}\label{theo:ia}
Let $H\subseteq\ia$ be a subgroup, and let $S$ be a free splitting of $F_N$ which is invariant by some finite-index subgroup of $H$. Then $S$ is $H$-invariant and $H$ acts as the identity on the quotient graph $S/F_N$ (in particular, all collapses of $S$ are also $H$-invariant).
\end{theo}

\subsection{Relative automorphism groups}

A \emph{free factor} of $F_N$ is a subgroup generated by a subset of a free basis of $F_N$. A \emph{free factor system} of $F_N$ is a finite collection $\calf=\{[A_1],[A_2], \ldots, [A_k]\}$ of conjugacy classes of free factors that arise as the collection of conjugacy classes of nontrivial point stabilizers in a nontrivial free splitting of $F_N$. 

A free factor system $\calf$ is \emph{sporadic} if either $\calf$ consists of the conjugacy class of a single corank one free factor, or else $\calf$ consists of the conjugacy classes of two free factors $A$ and $B$ satisfying $F_N=A\ast B$.

Given a free factor system $\calf$, we denote by $\Out(F_N,\calf)$ the subgroup of $\Out(F_N)$ made of all outer automorphisms that preserve all the conjugacy classes of subgroups in $\calf$.

A splitting $S$ of $F_N$ is said to be \emph{relative to $\calf$} if every subgroup of $F_N$ whose conjugacy class belongs to $\calf$ is elliptic in $S$.

An \emph{$(F_N,\calf)$-free factor} is a subgroup of $F_N$ that arises as a point stabilizer in some free splitting of $F_N$ relative to $\calf$. It is \emph{proper} if it is nontrivial, not conjugate to one of the subgroups of $\calf$, and not equal to $F_N$.

An element $g\in F_N$ is \emph{$\calf$-peripheral} if it is contained in a subgroup of $F_N$ whose conjugacy class belongs to $\calf$. 

The \emph{free factor graph} $\FF(F_N,\calf)$ is the graph whose vertices are the $F_N$-equivariant homeomorphism classes of free splittings of $F_N$ relative to $\calf$, where two splittings are joined by an edge if they are compatible or share a common non-$\calf$-peripheral elliptic element. We refer the reader to \cite[Section~2.2.1]{GH1} for a comparison with other quasi-isometric models in the literature. By \cite{BF,HM,GH1}, the free factor graph $\FF(F_N,\calf)$ is hyperbolic.

The following fact is of central importance in the present paper. It essentially relies on \cite[Proposition~5.1]{GH1}, together with the theorem of Handel--Mosher \cite[Theorem~3.1]{HM2} mentioned above. See \cite[Proposition~2.5]{HW}.

\begin{prop}\label{prop:unbounded-orbits-in-ff}
Let $H\subseteq\ia$ be a subgroup, and let $\calf$ be a maximal $H$-invariant free factor system. Assume that $\calf$ is nonsporadic.

Then the $H$-action on $\FF(F_N,\calf)$ has unbounded orbits.
\end{prop}

The Gromov boundary of $\FF(F_N,\calf)$ has been described in terms of relatively arational trees \cite{BR,Ham,GH1}. An \emph{$(F_N,\calf)$-tree} is an $\mathbb{R}$-tree equipped with an $F_N$-action by isometries, in which every subgroup of $F_N$ whose conjugacy class belongs to $\calf$ is elliptic. An $(F_N,\calf)$-tree $T$ is \emph{arational} if no proper $(F_N,\calf)$-free factor acts elliptically on $T$, and for every proper $(F_N,\calf)$-free factor $A$, the $A$-minimal invariant subtree of $T$ is a simplicial $A$-tree in which every nontrivial point stabilizer is conjugate into one of the subgroups in $\calf$. Two arational $(F_N,\calf)$-trees are \emph{equivalent} if they are $F_N$-equivariantly homeomorphic when equipped with the \emph{observers' topology}: this is the topology on a tree $T$ for which a basis of open sets is given by the connected components of the complements of points in $T$.  

\begin{theo}[{\cite{BR,Ham,GH1}}]
Let $\calf$ be a nonsporadic free factor system of $F_N$. The Gromov boundary of $\FF(F_N,\calf)$ is $\Out(F_N,\calf)$-equivariantly homeomorphic to the space of all equivalence classes of arational $(F_N,\calf)$-trees.
\end{theo}

We also require the following fact, which relies on the description of equivalence classes of arational trees \cite[Proposition~13.5]{GH2}. 

\begin{lemma}\label{lemma:boundary-tree}
Let $\calf$ be a nonsporadic free factor system of $F_N$, and let $H\subseteq\Out(F_N,\calf)$ be a subgroup. If $H$ fixes a point in $\partial_\infty\FF(F_N,\calf)$, then $H$ has a finite-index subgroup that fixes the homothety class of an arational $(F_N,\calf)$-tree.
\end{lemma}

\subsection{Twists associated to splittings} \label{sec:twists-background}

In the present paper, we will often take advantage of automorphisms that preserve certain splittings of $F_N$. Stabilizers of splittings have been extensively studied by Bass--Jiang \cite{BJ}, and Levitt \cite{Lev}.  The stabilizer of every splitting has a subgroup called the \emph{group of twists}. Roughly speaking, each separating half-edge has an associated group of partial conjugations in $\Out(F_N)$, and each nonseparating half-edge has an associated group of transvections. We now review the precise definition.

Let $S$ be a splitting of $F_N$, let $v\in S$ be a vertex and let $e$ be a half-edge of $S$ incident on $v$. We denote by $G_v$ the $F_N$-stabilizer of $v$ and by $G_e$ the $F_N$-stabilizer of $e$. Assume that $G_e$ is either trivial or cyclic, and let $z\in G_v$ be an element that centralizes $G_e$. The \emph{twist by $z$ around $e$} is the automorphism $D_{e,z}$ of $F_N$ (stabilizing $S$) defined as follows. Let $\overline{S}$ be the splitting obtained from $S$ by collapsing all half-edges outside of the orbit of $e$. Let $\overline{v}$ and $\overline{e}$ be the images of $v$ and $e$ in $\overline{S}$, respectively, and let $\overline{w}$ be the other extremity of $\overline{e}$. 

If $\overline{v}$ and $\overline{w}$ belong to distinct $F_N$-orbits, then $\overline{S}$ determines a decomposition of $F_N$ as an amalgam, and $D_{e,z}$ is defined as the (unique) automorphism that acts as the identity on $G_{\overline{v}}$, and as conjugation by $z$ on $G_{\overline{w}}$. 

If $\overline{v}$ and $\overline{w}$ belong to the same $F_N$-orbit, then $\overline{S}$ determines a decomposition of $F_N$ as an HNN extension. We then let $t\in F_N$ be such that $\overline{w}=t\overline{v}$, and $D_{e,z}$ is defined as the (unique) automorphism that acts as the identity on $G_{\overline{v}}$ and sends $t$ to $zt$ -- notice that this definition does not depend on the choice of $t$ as above.

The \emph{group of twists} of the splitting $S$ is the subgroup of $\Out(F_N)$ generated by all twists around half-edges of $S$. 

\begin{ex}
The following will be the most crucial example throughout the paper. Let $S$ be a one-edge nonseparating free splitting of $F_N$. Then $S$ is the Bass--Serre tree of a decomposition of $F_N$ as an HNN extension $F_N=A\ast$ for some corank one free factor $A$ of $F_N$. The group of twists of the splitting $S$ decomposes as a direct product $K_1\times K_2$, where each $K_i$ is a nonabelian free group corresponding to the group of twists about one of the two half-edges of the quotient graph $S/F_N$. Choosing a stable letter gives a natural way to identify each $K_i$ with $A$. Algebraically, let $t$ be a stable letter for the splitting. Then the elements in the group of twists of $S$ are precisely the outer automorphisms of $F_N$ which have a representative in $\Aut(F_N)$ acting as the identity on $A$ and sending $t$ to $a_1ta_2$ with $a_1,a_2\in A$. 
\end{ex}

We can see how the group of twists fits inside the stabilizer of the splitting with the following result of Levitt:

\begin{proposition}[{\cite[Proposition~4.1]{Lev}}]\label{prop:Levitt} 
Let $S$ be a free splitting of $F_N$ and let $V$ be a set of representatives of $F_N$-orbits of vertices in $S$. The subgroup $\Stab^0(S)$ of the $\Out(F_N)$-stabilizer of $S$ made of all automorphisms that act trivially on the quotient graph $S/F_N$ fits in the exact sequence \[ 1 \to \mathcal{T} \to \stab^0(S) \to \oplus_{v \in V} \Out(G_v) \to 1, \] 
where $\mathcal{T}$ is the group of twists of the splitting and $G_v$ is the stabilizer of the vertex $v \in V$. The group of twists is isomorphic to \[ \mathcal{T} \cong \oplus_{v \in V} (G_v)^{d_v}/Z(G_v), \] where the center $Z(G_v)$ of $G_v$ is embedded diagonally in $(G_v)^{d_v}$. \end{proposition}

\subsection{Product rank}

Following \cite[Section~6]{HW}, we define the \emph{product rank} of a group $H$, denoted as $\rk_{\pro}(H)$, to be the maximal integer $k$ such that a direct product of $k$ nonabelian free groups embeds in $H$. The following estimate regarding product rank and group extensions will be crucial throughout the paper.

\begin{lemma}[{\cite[Lemma~6.3]{HW}}]\label{lemma:rank-product-extension}
Let $1\to N\to G\to Q\to 1$ be a short exact sequence of groups. Then $\rk_{\pro}(G)\le\rk_{\pro}(N)+\rk_{\pro}(Q)$.
\end{lemma}

We say that two one-edge nonseparating free splittings are \emph{rose compatible} if they admit a common refinement $U$ such that $U/F_N$ is a two-petal rose. The following statement gathers results established in \cite[Theorem~6.1]{HW} and in \cite[Proposition~5.1]{BW}.

\begin{theo}\label{theo:product-rank-out-aut} 
The following hold.
\begin{enumerate}
\item For every $N \geq 2$, the product rank of $\aut(F_N)$ is $2N-3$. For every $N \ge 3$, the product rank of $\Out(F_N)$ is $2N-4$. 
\item Every subgroup of $\ia$ isomorphic to a direct product of $2N-4$ nonabelian free groups fixes a one-edge nonseparating free splitting of $F_N$.
\item Let $G \subseteq \Out(F_N)$ with $\rk_{\pro}(G)=2N-4$. Then any two one-edge free splittings fixed by $G$ are nonseparating and rose compatible. 
\end{enumerate}
\end{theo}

\section{Strongly rigid graphs and rigidity of subgroups}\label{sec:blueprint}

\emph{In this section we give a general framework to show that a subgroup $\G$ is rigid in a group $G$,  provided that one has what we call a `strongly rigid' action of $G$ on a graph.}
\\

\begin{de}
Let $G$ be a group. A subgroup $\Gamma\subseteq G$ is \emph{rigid in $G$} if every injective homomorphism $\Gamma\to G$ is conjugate to the inclusion of $\Gamma$ into $G$.
\end{de}

The goal of the present section is to present a general criterion that ensures that $\Gamma$ is rigid in $G$ (Proposition~\ref{prop:blueprint} below). We start with a few more definitions.

\begin{de}[Decorated $G$-graph]
Let $G$ be a group. A \emph{$G$-graph} is a graph equipped with an action of $G$ by graph automorphisms. A \emph{decorated $G$-graph} is a pair $(X,\calc)$, where $X$ is a $G$-graph and $\calc$ is a collection of subsets of the edge set of $X$ that are invariant under the $G$-action. We say that a graph map $\theta$ \emph{preserves the decoration} if $\theta(E) \subseteq E$ for every set of edges $E \in \calc$. 
\end{de}

The decoration $\calc$ can be thought of a coloring of the edge set of $X$, although in this generality we allow some edges to have no color, some edges to have multiple colors, and vertices may have many outgoing edges of the same color. 

\begin{de}[Strongly rigid decorated $G$-graph]
Let $G$ be a group. A $G$-graph $X$ is \emph{strongly rigid} if every injective  graph map $\theta : X \to X$ coincides with the action of a unique element of $G$. A decorated $G$-graph $(X,\calc)$ is \emph{strongly rigid} if every injective graph map $\theta : X \to X$ preserving $\calc$ coincides with the action of a unique element of $G$. 
\end{de}

If a $G$-graph $X$ is strongly rigid then $X$ is strongly rigid as a decorated $G$-graph with the trivial decoration $\calc=\emptyset$. 

\begin{remark}\label{rk:faithful} 
The $\out(F_N)$-graphs we consider in this paper are \emph{simple}: they do not have loops or multiple edges between vertices. Maps between simple graphs are determined by where they send vertices. In particular, if $X$ is strongly rigid then $G$ acts faithfully on the vertex set of $X$.
\end{remark}

\begin{prop}\label{prop:blueprint}
Let $G$ be a group, let $\Gamma\subseteq G$ be a subgroup and let $(X,\calc)$ be a simple, strongly rigid decorated $G$-graph. Assume that for every injective homomorphism $f:\Gamma\to G$, there exists an injective graph map $\theta: X \to X$ preserving the decoration and a normal subgroup $\Gamma'\unlhd \Gamma$ such that for every vertex $v \in X$, the following hold: 
\begin{itemize}
\item $f(\Stab_{\Gamma'}(v))\subseteq\Stab_G(\theta(v))$, and
\item if $f(\Stab_{\Gamma'}(v))\subseteq\Stab_G(w)$ then $w = \theta(v)$ (in other words, $\theta(v)$ is unique).
\end{itemize}
Then $\Gamma$ is rigid in $G$.
\end{prop}

\begin{proof}
	Let $f: \Gamma \to G$ be an injective homomorphism and let $\theta$ be the map given by our hypothesis. As $\theta$ is injective and preserves the decoration and $(X,\calc)$ is strongly rigid, there exists $g\in G$ so that for all $v\in X$, one has $\theta(v) = gv$.
	Conjugating $f$ by $g^{-1}$ yields a new map $f' = \ad_{g^{-1}}\circ f$ with the property that for every $v\in X$, one has
	\[ f'(\Stab_{\Gamma'}(v)) \subseteq g^{-1}\Stab_G(\theta(v))g = g^{-1}\Stab_G(gv)g = \Stab_G(v). \]
	Hence, it suffices to show the proposition in the case where $\theta = \id$.	In this case, observe that for every $\gamma\in \Gamma$ and every $v\in X$, we have
	\[ f(\Stab_{\Gamma'}(\gamma v)) \subseteq \Stab_G(\gamma v). \]
	On the other hand, using normality of $\Gamma'$ inside $\Gamma$, we have
	$$f(\Stab_{\Gamma'}(\gamma v))  = f(\gamma\Stab_{\Gamma'}(v)\gamma^{-1}),$$
from which we deduce that	
	\begin{displaymath}
	\begin{array}{rl}
	f(\Stab_{\Gamma'}(\gamma v)) & = f(\gamma)f(\Stab_{\Gamma'}(v))f(\gamma)^{-1} \\
	& \subseteq f(\gamma)\Stab_G(v)f(\gamma)^{-1} \\
	& = \Stab_G(f(\gamma)v).
	\end{array}
	\end{displaymath}
	The uniqueness assumption for $\theta$ therefore ensures that for all $\gamma\in \Gamma$ and all $v\in X$, one has
	\[ \gamma v = f(\gamma)v. \]
	Since $G$ acts on the vertex set of $X$ faithfully (see Remark~\ref{rk:faithful}),
	this implies that $f(\gamma)=\gamma$ for all $\gamma\in \Gamma$, which concludes our proof.
\end{proof}

\section{Strong rigidity of an $\Out(F_N)$-graph}\label{sec:complexes}

\emph{In this section we show that for $N \geq 4$, the edgewise nonseparating free splitting graph is stongly rigid as an $\Out(F_N)$-graph. When $N=3$, we show that a decorated version of the nonseparating free splitting graph is also strongly rigid.}
\\

The \emph{nonseparating free splitting graph}  is the graph $\ns$  whose vertices are the one-edge nonseparating free splittings of $F_N$, where two splittings are joined by an edge if they have a common refinement. Adjacent splittings in $\ns$ are either \emph{rose compatible} (the quotient graph of their common refinement is a two-petal rose), or are \emph{circle compatible} (the quotient graph of their common refinement is a two-edge loop). 

The \emph{edgewise nonseparating free splitting graph}  is the subgraph $\ens$ of $\ns$ with the same vertex set (the one-edge nonseparating free splittings of $F_N$), however we only add an edge between two splittings if they are rose compatible. In \cite{HW}, the last two named authors proved the following statement.

\begin{theo}[{\cite[Theorem~3.5]{HW}}]\label{theo:auto-ens}
For every $N\ge 3$, the natural map $$\Out(F_N)\to\Aut(\ens)$$ is an isomorphism.
\end{theo}

The goal of the present section is to prove the following theorem.

\begin{theo}\label{theo:graph-maps-ens}
For every $N \geq 4$, every injective graph map of $\ens$ is an automorphism.
\end{theo}

We also establish a variation that will enable us to carry arguments in rank $3$.

\begin{theo}\label{theo:graph-maps-ens-2}
Let $N\ge 3$, and let $f$ be an injective graph map of $\ens$ which extends to a graph map of $\ns$. 

Then $f$ is an automorphism.
\end{theo}

In the context of the previous section, the above results can be rephrased as follows: 

\begin{corollary}\label{cor:strongly-rigid} 
For every $N\ge 4$, the edgewise nonseparating free splitting graph $\ens$ is strongly rigid with respect to the action of $\Out(F_N)$. 

Let $\calc_r=\{E_r\}$ be the decoration of $\ns$ given by the edges in $\ens$ (i.e. the edges between rose compatible splittings). When $N=3$, the decorated $\Out(F_3)$-graph $(\ns,\calc_r)$ is strongly rigid.  
\end{corollary}

\begin{proof}
The first statement follows from Theorems~\ref{theo:auto-ens} and~\ref{theo:graph-maps-ens}. For the second statement, if $f: \ns \to \ns$ is an injective graph map preserving $\calc_r$, then $f$ restricts to an injective graph map of $\ens$. As $N=3$, the result follows by combining Theorems~\ref{theo:auto-ens} and~\ref{theo:graph-maps-ens-2}.
\end{proof}

The proofs of Theorems~\ref{theo:graph-maps-ens} and \ref{theo:graph-maps-ens-2} follow the same strategy. We show that injective graph maps send $N$-petal roses to $N$-petal roses, and then use the local finiteness of the spine of Outer space to prove surjectivity. We will first show that incompatibility is preserved for certain pairs of splittings that we call \emph{cagey pairs}; these are defined and studied in the next section.

\subsection{Cagey pairs are sent to cagey pairs.}

In the rest of this section, we will always assume that $N\ge 3$ is fixed. We fix once and for all an identification of $F_N$ with the fundamental group of a doubled handlebody $M_N:=\sharp^NS^1\times S^2$. Under this identification, one-edge free splittings of $F_N$ are in one-to-one correspondence with homotopy classes of essential embedded $2$-spheres in $M_N$. More generally, if $\Sigma$ is a collection of pairwise disjoint, pairwise nonhomotopic spheres in $M_N$, one obtains a splitting $T_\Sigma$ by taking the dual tree to the lifts $\tilde \Sigma$  of the spheres in the universal cover of $M_N$. The map $\Sigma \mapsto T_\Sigma$ induces a bijection between homotopy classes of sphere systems in $M_N$ and free splittings of $F_N$ (see \cite[Lemma~2]{AS}). In the same way that we consider free splittings equivalent under  equivariant homeomorphism, we usually consider sphere systems in $M_N$ up to homotopy. The vertices in the quotient graph $T/F_N$ correspond to the connected components of $M_N - \Sigma$, and the elements of $\Sigma$ correspond to the edges of $T/F_N$. 

In this section we will use lower case notation such as $\sigma$ and $\tau$ to denote single spheres or one-edge splittings, and upper case notation such as $S$ for arbitrary splittings (equivalently collections of pairwise disjoint spheres). Due to the above correspondence, we shall sometimes blur the distinction between essential spheres in $M_N$ and their associated one-edge splittings of $F_N$. For any sphere $\sigma$ the one-edge free splitting $T_\sigma$ is nonseparating if and only if the manifold $M- \sigma$ is connected.

Two essential embedded $2$-spheres $\sigma$ and $\tau$ in $M_N$ that intersect in a single essential circle have an associated \emph{boundary splitting}. This is obtained by taking the boundary of the regular neighbourhood of the union $\sigma \cup \tau$, or equivalently cutting $\sigma$ and $\tau$ into halves $\sigma_1, \sigma_2$ and $\tau_1, \tau_2$ along the circle of intersection and taking the spheres corresponding to the unions $\sigma_i\cup \tau_j$ for $i,j \in \{1,2\}$. When $\sigma$ and $\tau$ are nonseparating the boundary splitting contains either 3 or 4 distinct spheres, and there are six different possibilities for the splitting. These are depicted in Figure~\ref{fig:boundary-spheres}, which was also given in \cite[Section~3]{HW}.

\begin{figure}
\centering \def\svgwidth{300pt} 
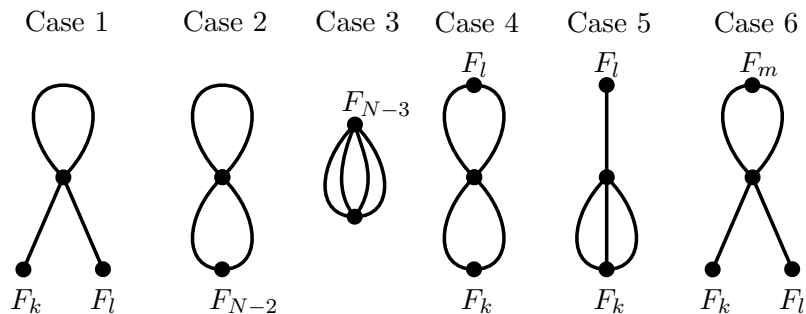
\caption{The six possibilities for the (quotient graph of groups of) the boundary splitting for spheres $\sigma$ and $\tau$ that intersect in a single circle. The splittings $\sigma$ and $\tau$ give refinements at the unlabelled vertex in each case, which has trivial stabilizer.}
\label{fig:boundary-spheres}
\end{figure}

A \emph{cage} is a bipartite free splitting $S$ of $F_N$ with two orbits of vertices such that every edge intersects both vertex orbits. Equivalently, the quotient graph $S/F_N$ is a graph with two vertices and no loop edges. A one-edge separating free splitting is a (slightly degenerate) example of a cage. 

\begin{definition}[Cagey pairs]
A pair $(\sigma, \tau)$ of essential embedded nonseparating $2$-spheres in $M_N$ is \emph{cagey} if $\sigma$ and $\tau$ intersect in a single essential circle, and the boundary splitting of the pair $(\sigma, \tau)$ is a cage (Case 3 in Figure~\ref{fig:boundary-spheres}).
\end{definition}

 The following lemma gives a characterization of cagey pairs given purely in terms of the combinatorics of the graph $\ens$; a direct consequence (Corollary~\ref{cor:cagey-pairs-preserved} below) will be that cagey pairs are sent to cagey pairs under injective graph maps of $\ens$.

\begin{lemma}\label{l:cagey_pairs_preserved}
Let $\sigma$ and $\tau$ be two essential nonseparating embedded $2$-spheres in $M_N$. The pair $(\sigma,\tau)$ is cagey if and only if there exists a clique $\Sigma$ in $\ens$ such that $\Sigma \cup \tau$ and $\Sigma \cup \sigma$ are both (3N-3)-cliques in $\ens$ (in particular, they are maximal). 
\end{lemma}

\begin{proof}
If $(\sigma,\tau)$ is a cagey pair, then one can blow up the boundary splitting of the pair $(\sigma,\tau)$ (which is a cage) to get a system $\Sigma$ such that $\Sigma \cup \tau$ and $\Sigma \cup \sigma$ are both (3N-3)-cliques in $\ens$, see \cite[Figure~3]{HW}. 

Conversely, suppose that such a system $\Sigma$ exists. Then the splitting associated to $\Sigma$ contains  $3N-4$ orbits of edges, and $\sigma$ and $\tau$ correspond to blow-ups of this splitting at its unique valence 4 vertex. This implies that $\sigma$ and $\tau$ intersect in a single essential circle.  Let $S$ be the boundary splitting of $\sigma$ and $\tau$. Then every sphere disjoint from both $\sigma$ and $\tau$ is also disjoint from every sphere in $S$. Maximality of $\sigma \cup \Sigma$ thus implies that every edge of the boundary splitting $S$ is contained in $\Sigma$. As any two one-edge collapses of $S$ are pairwise adjacent in $\ens$, an inspection of the possible types of boundary splitting given in Figure~\ref{fig:boundary-spheres} shows that $S$ has to be a cage (for all other potential boundary splittings, the quotient graph either  contains a separating edge or a pair of edges whose union separates). Hence $(\sigma, \tau)$ is a cagey pair.  
\end{proof}

\begin{cor}\label{cor:cagey-pairs-preserved}
Let $N\ge 3$, and let $f$ be an injective graph map of $\ens$. Let $\sigma,\tau$ be two essential nonseparating embedded $2$-spheres.

If the pair $(\sigma,\tau)$ is cagey, then so is the pair $(f(\sigma),f(\tau))$.
\end{cor}

\begin{proof}
The property of two spheres containing a common clique of size $3N-4$ in their link is preserved under the map $f$, so the conclusion follows from Lemma~\ref{l:cagey_pairs_preserved}.
\end{proof}

\subsection{Roses are sent to roses.}

If $f$ is an injective graph map of $\ens$, then a clique $S$ in $\ens$ is sent to another clique $f(S)$. However, it requires work to show that the splitting determined by $S$ has the same quotient graph $S/F_N$ to the splitting given by $f(S)$. A key step in our proof of Theorem~\ref{theo:graph-maps-ens} is to show the following statement.

\begin{prop}\label{p:roses_sent_to_roses}
Let $N\ge 3$, and let $f$ be an injective graph map of $\ens$. Assume that either $N\ge 4$, or else that $f$ extends to an injective graph map of $\ns$. 

Let $S=\{\sigma_1,\dots,\sigma_N\}$ be a collection of one-edge nonseparating free splittings such that the quotient graph of the splitting given by $S$ is an $N$-petal rose.

Then $f(\sigma_1),\dots,f(\sigma_N)$ are compatible and refine to a splitting whose quotient graph is an $N$-petal rose. 
\end{prop}

This is the most involved part of the strong rigidity statement. We will use the language of sphere systems following Hatcher~\cite{Hat}, coupled with the combinatorial methods for describing blow-ups of graphs developed by Culler and Vogtmann~\cite{CV}. The latter is particularly useful for describing examples precisely and studying boundary splittings of intersecting spheres.

\subsubsection{Blow-ups of splittings via ideal edges}\label{sec:blowup}

In this subsection we fix a set $\Sigma=\{x_1, \ldots, x_N\}$ of compatible one-edge free splittings that refine to an $N$-rose. Every free splitting $\tau$ disjoint from $\Sigma$ corresponds to a partition of the set of directions $D=\{x_1^+, x_1^-, \ldots, x_N^+, x_N^-\}$ at the vertex of the rose into two non-empty sets $\calp=P_1 \cup P_2$. If each set $P_i$ in the partition contains at least two elements then the partition is \emph{thick} and $\tau$ gives a one-edge blow-up of the rose. Combinatorially, the blow-up is obtained by attaching the half-edges in $P_1$ to one vertex of the new edge and the half-edges in $P_2$ to the other vertex. In the language of sphere systems, the cut manifold $M_N- \Sigma$ has $2N$ boundary components corresponding to $D$ and $\tau$ is the sphere in $M_N -\Sigma$ that partitions the boundary components according to $\calp$ (such a sphere is unique up to homotopy).  Each petal $x_i$ of the rose is represented by two partitions, one of which contains the singleton set $\{x_i^+\}$, and the other contains the singleton set $\{x_i^-\}$.  In this section, sphere systems are often convenient for visualizing splittings in diagrams whereas the combinatorial description via partitions is useful for precise proofs.

\begin{definition}
An \emph{ideal edge} is a partition $\calp$ of $\{x_1^+, x_1^-, \ldots, x_N^+, x_N^-\}$ separating some pair $\{x_i^+,x_i^-\}$.
\end{definition}

\begin{remark}
In this paper, ideal edges are very convenient as they are exactly the partitions that determine nonseparating splittings (see, for example, \cite[Section~8]{Vog}). Ideal edges in \cite[Section~2]{CV} were required to be thick, whereas \cite{Vog} admitted \emph{trivial ideal edges} corresponding to the petals of the rose. For our purposes we will also consider petals as ideal edges.
\end{remark}

 Following the language of \cite{CV}, two partitions $\calp=\{P_1,P_2\}$ and $\calq=\{Q_1,Q_2\}$ \emph{cross} if $P_i\cap Q_j \neq \emptyset$ for all $i,j \in \{1,2\}$ and are compatible otherwise (i.e. together they give a two-edge splitting). Compatibility is equivalent to a choice of indexing of the sides of $\calp$ and $\calq$ such that $P_1$ and $Q_1$ are disjoint. A maximal blow-up of the rose is given by a set of $2N-3$ distinct, pairwise compatible thick partitions. If all these partitions are ideal edges, then the corresponding graph has no separating edges and determines a maximal clique in $\ns$.

\begin{lemma}[Rose compatibility of ideal edges]\label{l:combinatorial_conditions_1}
Let $\calp=\{P_1,P_2\}$ and $\calq=\{Q_1,Q_2\}$ be two ideal edges of an $N$-rose. 

If $\calp$ and $\calq$ are compatible, and the sides of these partitions are chosen so that $P_1$ and $Q_1$ are disjoint, then $\calp$ and $\calq$ are rose compatible (i.e. they span an edge in $\ens$) if and only if $P_1 \cup Q_1$ separates some pair $\{x_i^+, x_i^-\}$.
\end{lemma}

\begin{proof}
We first look at the most general case, which is when neither $\mathcal{P}$ or $\mathcal{Q}$ is a petal of the rose. Then the  joint blow-up $\Gamma$ of the rose by $\calp$ and $\calq$ has a maximal tree $T$ given by $\calp$ and $\calq$, which is just a two-edge line. We label the endpoints of this line $p$ and $q$ and label the central vertex $v$.  If the partitions are chosen so that $P_1$ and $Q_1$ are disjoint, then one finds the full blow-up $\Gamma$ by attaching the edges $x_1, \ldots, x_N$ to $T$ like so; we attach the half-edges in $P_1$  to the point $p$, the half-edges in $Q_1$ to the point $q$, and the remaining half-edges to the central vertex $v$. As $\calp$ is an ideal edge it separates some $\{x_i^+,x_i^-\}$, so there is an edge $x_i$  from $p$ to one of the other vertices in $T$. Similarly, as $\calq$ is an ideal edge there is an edge $x_j$ from $q$ to either $p$ or $v$. If we look at the different cases for how these edges can attach, the only time we attain a two-edge loop after collapsing all the $x_i$'s is when no outgoing edge from $p$ or $q$ is attached to $v$. This happens exactly when no pair $\{x_i^+,x_i^-\}$ is separated by $P_1 \cup Q_1$, otherwise $\calp$ and $\calq$ are rose compatible. If both $\mathcal{P}$ and  $\mathcal{Q}$ are petals of the rose then they are clearly rose compatible. Finally, we may assume that $\mathcal{P}$ is a petal and $\mathcal{Q}$ is a thick ideal edge. Then $\mathcal{P}$ and $\mathcal{Q}$ are rose-compatible unless the blow-up of the rose by $\mathcal{Q}$ has the ends of $\mathcal{P}$ on either side and does not split any of the ends of the other loops. Without loss of generality, if $P_1=\{x_1^+\}$, then this happens if and only if $Q_1$ consists of $x_1^-$ and a collection of $\{x_j^+, x_j^{-}\}$ pairs, which happens if and only if $P_1 \cup Q_1$ does not separate any pair $\{x_i^+, x_i^-\}$.
\end{proof}

If $\calp$ and $\calq$ cross, then their corresponding spheres intersect in a single circle, and the boundary splitting associated to these spheres is given by the sets of the form $P_i\cap Q_j$ (see Figure~\ref{f:boundary}).

\begin{figure}[ht]  \centering \def\svgwidth{200pt} 
%% Creator: Inkscape inkscape 0.92.3, www.inkscape.org
%% PDF/EPS/PS + LaTeX output extension by Johan Engelen, 2010
%% Accompanies image file 'boundary.pdf' (pdf, eps, ps)
%%
%% To include the image in your LaTeX document, write
%%   \input{<filename>.pdf_tex}
%%  instead of
%%   \includegraphics{<filename>.pdf}
%% To scale the image, write
%%   \def\svgwidth{<desired width>}
%%   \input{<filename>.pdf_tex}
%%  instead of
%%   \includegraphics[width=<desired width>]{<filename>.pdf}
%%
%% Images with a different path to the parent latex file can
%% be accessed with the `import' package (which may need to be
%% installed) using
%%   \usepackage{import}
%% in the preamble, and then including the image with
%%   \import{<path to file>}{<filename>.pdf_tex}
%% Alternatively, one can specify
%%   \graphicspath{{<path to file>/}}
%% 
%% For more information, please see info/svg-inkscape on CTAN:
%%   http://tug.ctan.org/tex-archive/info/svg-inkscape
%%
\begingroup%
  \makeatletter%
  \providecommand\color[2][]{%
    \errmessage{(Inkscape) Color is used for the text in Inkscape, but the package 'color.sty' is not loaded}%
    \renewcommand\color[2][]{}%
  }%
  \providecommand\transparent[1]{%
    \errmessage{(Inkscape) Transparency is used (non-zero) for the text in Inkscape, but the package 'transparent.sty' is not loaded}%
    \renewcommand\transparent[1]{}%
  }%
  \providecommand\rotatebox[2]{#2}%
  \newcommand*\fsize{\dimexpr\f@size pt\relax}%
  \newcommand*\lineheight[1]{\fontsize{\fsize}{#1\fsize}\selectfont}%
  \ifx\svgwidth\undefined%
    \setlength{\unitlength}{145.27962392bp}%
    \ifx\svgscale\undefined%
      \relax%
    \else%
      \setlength{\unitlength}{\unitlength * \real{\svgscale}}%
    \fi%
  \else%
    \setlength{\unitlength}{\svgwidth}%
  \fi%
  \global\let\svgwidth\undefined%
  \global\let\svgscale\undefined%
  \makeatother%
  \begin{picture}(1,0.59586232)%
    \lineheight{1}%
    \setlength\tabcolsep{0pt}%
    \put(0,0){\includegraphics[width=\unitlength,page=1]{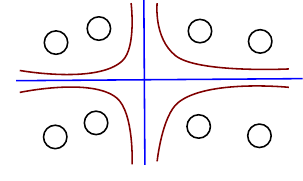}}%
    \put(-0.00214262,0.32928802){\color[rgb]{0,0,0}\makebox(0,0)[lt]{\lineheight{1.25}\smash{\begin{tabular}[t]{l}$\calp$\end{tabular}}}}%
    \put(0.47530042,0.00536915){\color[rgb]{0,0,0}\makebox(0,0)[lt]{\lineheight{1.25}\smash{\begin{tabular}[t]{l}$\calq$\end{tabular}}}}%
    \put(0.00966691,0.47775087){\color[rgb]{0,0,0}\makebox(0,0)[lt]{\lineheight{1.25}\smash{\begin{tabular}[t]{l}$P_1$\end{tabular}}}}%
    \put(0.01978938,0.15551906){\color[rgb]{0,0,0}\makebox(0,0)[lt]{\lineheight{1.25}\smash{\begin{tabular}[t]{l}$P_2$\end{tabular}}}}%
    \put(0.22223879,0.01211751){\color[rgb]{0,0,0}\makebox(0,0)[lt]{\lineheight{1.25}\smash{\begin{tabular}[t]{l}$Q_1$\end{tabular}}}}%
    \put(0.74354583,0.02055279){\color[rgb]{0,0,0}\makebox(0,0)[lt]{\lineheight{1.25}\smash{\begin{tabular}[t]{l}$Q_2$\end{tabular}}}}%
    \put(0,0){\includegraphics[width=\unitlength,page=2]{boundary.pdf}}%
  \end{picture}%
\endgroup%
 \caption{Crossing ideal edges $\calp$ and $\calq$ are given as blue lines, with the associated boundary spheres shown in red.} 
\label{f:boundary} 
\end{figure}

If $\calp$ and $\calq$ are nonseparating then this boundary splitting is cagey if and only if every pair of spheres in the boundary splitting span an edge in $\ens$ (this follows from the description of the possible boundary splittings given in Figure~\ref{fig:boundary-spheres}).

\begin{lemma}[A combinatorial description of cagey pairs] \label{l:combinatorial_conditions_2}
Suppose $\calp$ and $\calq$ are ideal edges in an $N$-rose that cross. Then $\calp$ and $\calq$ form a cagey pair if and only if 
\begin{itemize}\item the set $P_i \cap Q_j$ determines an ideal edge for all $i,j$ and,
 \item for each pair $K_{i,j}=P_i \cap Q_j$ and $K_{k,l}=P_k \cap Q_l$ with $(i,j) \neq (k,l)$, the set $K_{i,j} \cup K_{k,l}$ separates some pair of the form $\{x_m^+, x_m^-\}$.  
\end{itemize}

\end{lemma}

\begin{proof}
The condition that $\calp$ and $\calq$ form a cagey pair is equivalent to the set of spheres in the boundary splitting forming a clique in $\ens$. Each boundary sphere is nonseparating if and only if each set $K_{i,j}=P_i \cap Q_j$ determines an ideal edge. Furthermore, by Lemma~\ref{l:combinatorial_conditions_1}, these splittings form a clique in $\ens$ if and only if for each choice $(i,j)\neq (k,l)$ the union $K_{i,j} \cup K_{k,l}$ separates some pair $\{x_m^+,x_m^-\}$. 
\end{proof}

\begin{lemma}[Rigid blow-up lemma] \label{l:rigid_blow_up} 
Let $S=\{\sigma_1, \ldots, \sigma_N\}$ form an $N$-rose. There exists a sphere system $\Sigma$ and a sphere $\tau$ such that \begin{itemize}
\item $S \cup \Sigma$ is a maximal clique in $\ens$ (consisting of $3N-3$ spheres).
\item The sphere $\tau$ is adjacent to every element of $S$ in $\ens$ and forms a cagey pair with every element of $\Sigma$.
\end{itemize}
\end{lemma}

\begin{figure}[ht]  \centering \def\svgwidth{300pt} 
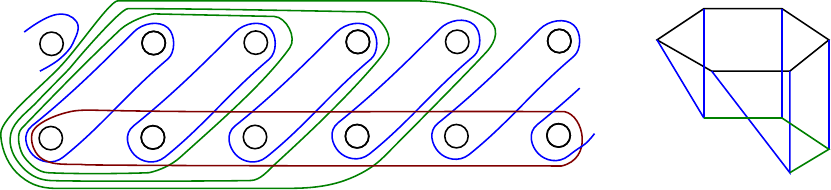 \caption{On the left is a description of the blow-up used in Lemma~\ref{l:rigid_blow_up} when $N=6$ in terms of spheres. The 12 black circles correspond to each side of the six spheres forming a rose, with $x_i^+$ and $x_i^-$ pictured vertically. The blue spheres correspond to the partitions $\calp^i$ and the green spheres correspond to the partitions $\calq^i$. The sphere $\tau$ is depicted in red. On the right is the graph given by the blow-up $S \cup \Sigma$.} 
\label{f:blow_up} 
\end{figure}

\begin{proof}
Let $D=\{x_1^-,x_1^+,\ldots, x_N^-, x_N^+\}$ be the set of directions at the unique vertex in the $N$-rose. We define partitions $\calp^k$ with $P^k_1=\{x_k^-, x^+_{k+1}\}$ and $P^k_2=D-P^k_1$ up to cyclic ordering of the petals, so that $P^N_1=\{x_N^-,x_1^{+}\}$. We may also define partitions $\mathcal{Q}^k$ by the sets \[ Q_1^k=\{x_1^-,x_2^+,x_2^-,\ldots, x_k^+,x_k^-,x_{k+1}^+ \} \] and $Q_2^k=D-Q_1^k$ for $k=2, \ldots, N-2$. One can show (through, e.g. Lemma~\ref{l:combinatorial_conditions_1}) that the partitions $\calp^1,\ldots, \calp^N$ and $\calq^2,\ldots, \calq^{N-2}$ are ideal edges and are pairwise rose compatible. We let $\Sigma$ be the collection of these splittings. In the case when $N=6$ the blow-up $S\cup \Sigma$ is depicted in Figure~\ref{f:blow_up}. Let $\tau$ be the splitting determined by the partition \[ \tau=\{T_1=\{x_1^-,\ldots, x_N^-\}, T_2=\{x_1^+,\ldots, x_N^+\} \}.\] Then $\tau$ is nonseparating (so determines an ideal edge) and adjacent to, or equivalently rose compatible with, each sphere in $S$ in $\ens$. Furthermore, $\tau$ forms a cagey pair with each $\calp^k$ and each $\calq^k$. We give more details for an ideal edge of the form $\calq^k$ and leave the proof for edges of the form $\calp^k$ to the reader. In this situation, we have: \begin{align*} T_1 \cap Q_1^k &=\{ x_1^-, \ldots, x_k^-\} \\ T_1 \cap Q_2^k &= \{x_{k+1}^-, \ldots, x_N^- \} \\ T_2 \cap Q_1^k &= \{ x_2^+,\ldots, x_{k+1}^+ \} \\ T_2 \cap Q_2^k &=\{x_1^+, x_{k+2}^+,\ldots, x_N^+ \} \end{align*} All of these sets determine ideal edges, and the union of any two of these sets separates either $\{x_1^+,x_1^-\}$ or $\{x_N^+,x_N^-\}$. By Lemma~\ref{l:combinatorial_conditions_2}, the boundary splitting of $\tau$ and $\calq^k$ (determined by these four sets) forms a clique in $\ens$, so that $\tau$ and $\calq^k$ form a cagey pair.
\end{proof}

The rigid blow-up lemma is the first step towards showing that roses are sent to roses:

\begin{proposition}\label{prop:higher-valence}
Let $N\ge 3$, and let $f$ be an injective graph map of $\ens$. 

Let $S=\{\sigma_1,\dots,\sigma_N\}$ be a collection of pairwise compatible one-edge nonseparating free splittings such that the quotient graph of the splitting given by $S$ is an $N$-petal rose.

Then $f(\sigma_1),\dots,f(\sigma_N)$ are compatible and refine to a splitting whose quotient graph contains at most one vertex of valence greater than three. 
\end{proposition}

\begin{proof}
The fact that $f(\sigma_1),\dots,f(\sigma_N)$ are compatible is a consequence of $f$ being a graph map. Let $\Sigma$ and $\tau$ be the splittings that are given to us by the rigid blow-up lemma. As $\tau$ forms a cagey pair with each element of $\Sigma$, the splitting $f(\tau)$ forms a cagey pair with every element of $f(\Sigma)$ by Corollary~\ref{cor:cagey-pairs-preserved}. In terms of spheres, $f(\tau)$ intersects every sphere in $f(\Sigma)$ and is disjoint from each sphere in $f(S)$. Hence each sphere in $f(\Sigma)$ lies in the same component of $M_N-f(S)$ as $f(\tau)$. Translating this to graphs, all of the blow-ups of $f(S)$ by elements of $f(\Sigma)$ appear at the same vertex $v$. As $f(S)\cup f(\Sigma)$ is maximal, this implies that every other vertex $w \neq v$ of $f(S)$ has valence 3.
\end{proof}

The next step is to show that the quotient graph associated to the splitting $f(S)$ does not contain any valence three vertices. Given the work above, this will imply that the splitting has only one vertex and is therefore a rose.

\subsubsection{The image of a rose does not contain any valence 3 vertex.}

A second technical point in our proof of Proposition~\ref{p:roses_sent_to_roses} is to show the following fact. This will be the only place in this section where the argument for $N=3$ differs from the general case. 

\begin{prop}\label{prop:valence-3-excluded}
Let $N\ge 3$, and let $f$ be an injective graph map of $\ens$. Assume that either $N\ge 4$, or else that $f$ extends to an injective graph map of $\ns$.

Let $\sigma_1,\sigma_2$ and $\sigma_3$ be three one-edge nonseparating free splittings of $F_N$ which refine to a $3$-petal rose. 

Then $f(\sigma_1),f(\sigma_2)$ and $f(\sigma_3)$ are compatible and their common refinement does not contain any vertex of valence $3$.
\end{prop}

In the same spirit as the rigid blow-up lemma above, we look for useful configurations of disjoint spheres and cagey pairs that are preserved under injective graph maps.

\begin{lemma}\label{lemma:3-rose}
Let $N\ge 3$, and let $\sigma_1,\sigma_2,\sigma_3$ be three one-edge nonseparating free splittings of $F_N$ which refine to a $3$-petal rose. Then there exist two one-edge nonseparating free splittings $\tau_1,\tau_2$ of $F_N$ such that
\begin{enumerate}
\item[$(P_1)$] the splittings $\tau_1$ and $\tau_2$ are rose compatible,
\item[$(P_2)$] for every $i\in\{1,2\}$, the splittings $\sigma_i$ and $\tau_i$ form a cagey pair,
\item[$(P_3)$] for every $i\in\{1,2\}$ and every $j\in\{1,2,3\}$ with $i\neq j$, the splittings $\tau_i$ and $\sigma_j$ are compatible.
\end{enumerate}
Furthermore, if $N \geq 4$, we can choose $\tau_1$ and $\tau_2$ so that for every $i\in\{1,2\}$ and every $j\in\{1,2,3\}$ with $i\neq j$, the splittings $\tau_i$ and $\sigma_j$ are rose compatible.
\end{lemma}

\begin{proof}
The proof when $N=3$ and $N=4$ is depicted in Figure~\ref{fig:lemma-3-rose}. We will pick a standard rose $S$ with petals $x_1, \ldots, x_N$ as in the previous section and take $\sigma_1$ and $\sigma_2$ to be the ideal edges given by the partitions $\calp^1$ and $\calp^2$ in the proof of the rigid blow-up lemma. Hence $P^1_1=\{x_1^-, x^+_{2}\}$ and $P^2_1=\{x_2^-, x_3^+\}$. We then take $\sigma_3=x_3$. These three splittings determine a $3$-petal rose and, as $\Out(F_N)$ acts transitively on the set of $3$-petal roses (and their petals), we suffer no loss in generality by completing the proof in this setting. 

When $N=3$, we define $\tau_1$ to be the splitting determined by the partition $\mathcal{Q}=\{Q_1,Q_2\}$ with $Q_1=\{x_2^+,x_3^-\}$ and $\tau_2$ to be the splitting determined by the partition $\mathcal{R}=\{R_1,R_2\}$ with $R_1=\{x_1^+, x_2^-\}$. Using the combinatorial conditions in Lemma~\ref{l:combinatorial_conditions_1} and Lemma~\ref{l:combinatorial_conditions_2}, 
one verifies that these splittings satisfy conditions $(P_1)$ to $(P_3)$. Note, however, that $\tau_1$ and $\sigma_2$ are disjoint but their common refinement is a circle splitting (the same is true of $\tau_2$ and $\sigma_1$). When $N \geq 4$ we can make use of the extra room and ensure that the final statement of the proposition holds by defining $\tau_1$ to be the splitting determined by the partition $\mathcal{Q}=\{Q_1,Q_2\}$ with $Q_1=\{x_2^+,x_4^-\}$ and $\tau_2$ to be the splitting determined by the partition $\mathcal{R}=\{R_1,R_2\}$ with $R_1=\{x_3^+, x_4^+\}$. In this case, for every $i\in\{1,2\}$ and every $j\in\{1,2,3\}$ with $i\neq j$, the splittings $\tau_i$ and $\sigma_j$ are rose compatible.
\end{proof}

\begin{figure}
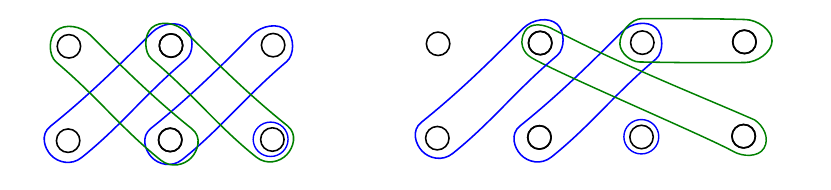 \caption{The proof of Lemma~\ref{lemma:3-rose} in the case of $N=3$ (on the left) and $N=4$ (on the right). In each case, the spheres $\sigma_1,\sigma_2$, and $\sigma_3$ are given in blue and the spheres $\tau_1$ and $\tau_2$ are given in green.}
\label{fig:lemma-3-rose}
\end{figure}

\begin{lemma}\label{lemma:valence-3-excluded}
Let $N\ge 3$. Let $\sigma_1,\sigma_2,\sigma_3$ be three pairwise compatible one-edge nonseparating free splittings of $F_N$. Assume that there exist two one-edge nonseparating free splittings $\tau_1,\tau_2$ of $F_N$ such that
\begin{enumerate}
\item[$(P_1)$] the splittings $\tau_1$ and $\tau_2$ are rose compatible,
\item[$(P_2)$] for every $i\in\{1,2\}$, the splittings $\sigma_i$ and $\tau_i$ form a cagey pair,
\item[$(P_3)$] for every $i\in\{1,2\}$ and every $j\in\{1,2,3\}$ with $i\neq j$, the splittings $\tau_i$ and $\sigma_j$ are compatible.
\end{enumerate}
Then $\sigma_1,\sigma_2$ and $\sigma_3$ refine to a splitting which does not contain any valence $3$ vertex.
\end{lemma}

\begin{figure}\centering \def\svgwidth{150pt} 
%% Creator: Inkscape inkscape 0.92.3, www.inkscape.org
%% PDF/EPS/PS + LaTeX output extension by Johan Engelen, 2010
%% Accompanies image file 'pants2.pdf' (pdf, eps, ps)
%%
%% To include the image in your LaTeX document, write
%%   \input{<filename>.pdf_tex}
%%  instead of
%%   \includegraphics{<filename>.pdf}
%% To scale the image, write
%%   \def\svgwidth{<desired width>}
%%   \input{<filename>.pdf_tex}
%%  instead of
%%   \includegraphics[width=<desired width>]{<filename>.pdf}
%%
%% Images with a different path to the parent latex file can
%% be accessed with the `import' package (which may need to be
%% installed) using
%%   \usepackage{import}
%% in the preamble, and then including the image with
%%   \import{<path to file>}{<filename>.pdf_tex}
%% Alternatively, one can specify
%%   \graphicspath{{<path to file>/}}
%% 
%% For more information, please see info/svg-inkscape on CTAN:
%%   http://tug.ctan.org/tex-archive/info/svg-inkscape
%%
\begingroup%
  \makeatletter%
  \providecommand\color[2][]{%
    \errmessage{(Inkscape) Color is used for the text in Inkscape, but the package 'color.sty' is not loaded}%
    \renewcommand\color[2][]{}%
  }%
  \providecommand\transparent[1]{%
    \errmessage{(Inkscape) Transparency is used (non-zero) for the text in Inkscape, but the package 'transparent.sty' is not loaded}%
    \renewcommand\transparent[1]{}%
  }%
  \providecommand\rotatebox[2]{#2}%
  \newcommand*\fsize{\dimexpr\f@size pt\relax}%
  \newcommand*\lineheight[1]{\fontsize{\fsize}{#1\fsize}\selectfont}%
  \ifx\svgwidth\undefined%
    \setlength{\unitlength}{238.64209994bp}%
    \ifx\svgscale\undefined%
      \relax%
    \else%
      \setlength{\unitlength}{\unitlength * \real{\svgscale}}%
    \fi%
  \else%
    \setlength{\unitlength}{\svgwidth}%
  \fi%
  \global\let\svgwidth\undefined%
  \global\let\svgscale\undefined%
  \makeatother%
  \begin{picture}(1,0.65946926)%
    \lineheight{1}%
    \setlength\tabcolsep{0pt}%
    \put(0,0){\includegraphics[width=\unitlength,page=1]{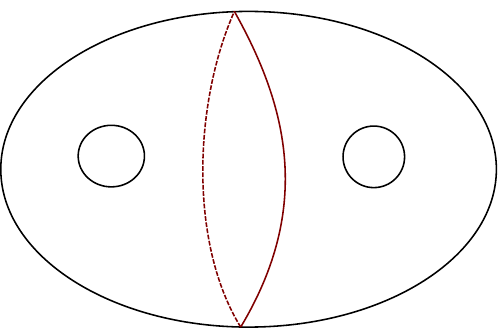}}%
    \put(0.06405098,0.59253166){\color[rgb]{0,0,0}\makebox(0,0)[lt]{\lineheight{1.25}\smash{\begin{tabular}[t]{l}$\sigma_1$\end{tabular}}}}%
    \put(0.19649689,0.22662197){\color[rgb]{0,0,0}\makebox(0,0)[lt]{\lineheight{1.25}\smash{\begin{tabular}[t]{l}$\sigma_2$\end{tabular}}}}%
    \put(0.74648361,0.22662179){\color[rgb]{0,0,0}\makebox(0,0)[lt]{\lineheight{1.25}\smash{\begin{tabular}[t]{l}$\sigma_3$\end{tabular}}}}%
    \put(0.44567427,0.31417099){\color[rgb]{0,0,0}\makebox(0,0)[lt]{\lineheight{1.25}\smash{\begin{tabular}[t]{l}$D_1$\end{tabular}}}}%
    \put(0,0){\includegraphics[width=\unitlength,page=2]{pants2.pdf}}%
  \end{picture}%
\endgroup%
 \caption{The proof of Lemma~\ref{lemma:valence-3-excluded}. The diagram shows that pants component $P$ with boundary spheres $\sigma_1$, $\sigma_2$, $\sigma_3$ and disk piece $D_1$ separating $\sigma_2$ from $\sigma_3$.}
\label{fig:pants}
\end{figure}

\begin{proof}
This proof is most easily understood in the language of sphere systems.  Suppose for a contradiction that the dual graph to the sphere system $\Sigma= \{\sigma_1,\sigma_2,\sigma_3\}$ has a trivalent vertex. In other words, $M_N - \Sigma$ has a component $P$ homeomorphic to $S^3$ with three balls removed (a three-dimensional version of a pair of pants). 
	
	Now consider the sphere system $T=\{\tau_1, \tau_2\}$. Following the terminology of \cite{Hat}, $T$ can be put normal form with respect to $\Sigma$ (one can either enlarge $\Sigma$ and $T$ to maximal systems or use the extension to nonmaximal systems given in \cite[Section~7.1]{HOP}).
	In particular, no component of $T-\Sigma$ (called a piece in the terminology of \cite{Hat}) is homotopic, fixing its boundary, into $\Sigma$. A piece of $T -\Sigma$ in $P$ is of one of three types: it is either a disk lying on one boundary component and separating the remaining two boundary components, an annulus between two distinct boundary components, or a genuine pair of pants with a boundary circle on each boundary sphere of $P$. 
	
	By assumptions $(P_2), (P_3)$ we have that $\tau_1$ is disjoint from $\sigma_2, \sigma_3$, and intersects
	$\sigma_1$ in a single circle. This implies that there is a single piece $D_1 = \tau_1 \cap P$ of $\tau_1$ 
	in $P$, and it has a single boundary component. Hence, $D_1$ is in fact a disk.
	
	By minimal position, $D_1$ separates $\sigma_2$ from $\sigma_3$ in $P$ (as depicted in Figure~\ref{fig:pants}). Namely, otherwise $D_1$ bounds a ball in $P$	together with a disk in $P$, which would allow to homotope $D_1$ (fixing its boundary) into $\sigma_1$.
	Since $P$ is homeomorphic to $S^3$ with three balls removed, each component of $P - D_1$ is therefore homeomorphic to
	the product $S^2\times(0,1)$.
	
	The sphere $\tau_2$ is disjoint from $D_1$, $\sigma_1$ and $\sigma_3$ and intersects  $\sigma_2$ in a single essential circle. Hence $\tau_2$ also has a single disk piece in $P$. However any disk in $P - D_1$ with boundary on $\sigma_2$ is homotopic fixing its boundary into $\sigma_2$. This contradicts $T$ being in normal form with respect to $\Sigma$ and proves the lemma.
\end{proof}

Armed with these two lemmas, we can now complete our proof of Proposition~\ref{prop:valence-3-excluded}.

\begin{proof}[Proof of Proposition~\ref{prop:valence-3-excluded}]
Let $N\ge 3$, and let $f$ be an injective graph map of $\ens$. Assume that either $N\ge 4$, or else that $f$ extends to an injective graph map of $\ns$. As $f$ preserves edges, the splittings $f(\sigma_1),f(\sigma_2)$ and $f(\sigma_3)$ are compatible.

By Lemma~\ref{lemma:3-rose}, there exist two one-edge nonseparating free splittings $\tau_1,\tau_2$ of $F_N$ such that the splittings $\sigma_i$ and $\tau_j$ satisfy Properties~$(P_1)$--$(P_3)$ from Lemma~\ref{lemma:3-rose}. Furthermore, if $N \geq 4$ we may choose $\tau_1$ and $\tau_2$ so that $\tau_i$ is rose compatible with $\sigma_j$ whenever $i \neq j$. Property~$(P_1)$ is preserved as $f$ is a graph map and property $(P_2)$ is preserved in view of Corollary~\ref{cor:cagey-pairs-preserved}. When $N=3$, property $(P_3)$ is preserved as $f$ extends to an injective graph map of $\ns$, and when $N\geq 4$ property $(P_3)$ is preserved as the splittings were chosen so that $\tau_i$ is rose compatible with $\sigma_j$ whenever $i \neq j$. So $f(\sigma_1),f(\sigma_2),f(\sigma_3),f(\tau_1),f(\tau_2)$ satisfy Properties~$(P_1)$--$(P_3)$ from Lemma~\ref{lemma:3-rose}, hence Lemma~\ref{lemma:valence-3-excluded} implies that the common refinement of $f(\sigma_1)$, $f(\sigma_2)$ and $f(\sigma_3)$ does not contain any valence $3$ vertex.
\end{proof}

We can finally show that $N$-roses are sent to $N$-roses in the desired setting. 

\begin{proof}[Proof of Proposition~\ref{p:roses_sent_to_roses}] 
Let $S=\{\sigma_1, \ldots, \sigma_N\}$ be a set of one-edge free splittings forming an $N$-rose. Let $f$ be an injective graph map of $\ens$, and assume that either $N\ge 4$, or else that $f$ extends to an injective graph map of $\ns$. By Proposition~\ref{prop:higher-valence}, the splittings $f(\sigma_1), \ldots, f(\sigma_N)$ are compatible and refine to a splitting whose quotient graph contains at most one vertex of valence greater than three.  If $f(S)$ contains a valence three vertex, then the three splittings corresponding to its adjacent edges are distinct (otherwise one would be separating). Hence some 3-element subset of $S' \subseteq S$ has the property that $f(S')$ contains a vertex of valence 3. This contradicts Proposition~\ref{prop:valence-3-excluded}. It follows that the splitting $f(S)$ determined by $f(\sigma_1), \ldots, f(\sigma_N)$ contains only one vertex, so is an $N$-rose.
\end{proof}

\subsection{Surjectivity via local finiteness}\label{s:surjectivity} 

Let $K$ be the graph whose vertices correspond to sets $\{\sigma_1,\dots,\sigma_N\}$ of pairwise compatible one-edge nonseparating free splittings refining to an $N$-rose, and whose edges correspond to sets $\{\sigma_1,\dots,\sigma_{N+1}\}$ of pairwise rose compatible  one-edge nonseparating free splittings so that $\sigma_1, \ldots, \sigma_N$ and $\sigma_1, \ldots, \sigma_{N-1}, \sigma_{N+1}$ both refine to an $N$-rose. In other words, an edge is given by an $(N+1)$-clique in $\ens$ with two distinct $N$-element rose subsets. The graph $K$ can be seen as the graph of roses in outer space which are adjacent if they have a `fat' one-edge common refinement (the two vertices of the refinement have to be connected by at least three edges in order for any two edges to be rose compatible). 

\begin{figure}[h]\centering 
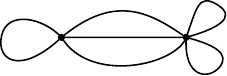 \caption{An example of a graph connecting adjacent roses in $K$ when $N=5$.}
\label{fig:Kgraph}
\end{figure}

\begin{lemma}
  The graph $K$ is connected and locally finite.
\end{lemma}

\begin{proof}
  Local finiteness is clear (in the complement of a filling sphere
  system there are only finitely many isotopy classes of spheres).
  Connectivity easily follows from the fact that Nielsen transformations together with the stabilizer of an $N$-petal rose generate $\Out(F_N)$, as was observed in \cite[p.~395]{BV} (if two roses are Nielsen adjacent in the language of \cite{BV}, then they are adjacent in $K$).
\end{proof}

Let $f$ be an injective graph map of $\ens$. Assume that either $N\ge 4$, or else that $f$ extends to a graph map of $\ns$. It follows from Proposition~\ref{p:roses_sent_to_roses} that $f$ induces an injective graph map $f_K$ of the graph $K$.

\begin{lemma}
Let $N\ge 3$, and let $f$ be an injective graph map of $\ens$. Assume that either $N\ge 4$, or else that $f$ extends to a graph map of $\ns$. 

Then $f_K$ is surjective.
\end{lemma}

\begin{proof}
  As all links of vertices of $K$ are isomorphic to one another and finite, local
  injectivity of $f_K$ implies local surjectivity of
  $f_K$. Connectivity of $K$ then implies surjectivity of $f_K$.
\end{proof}

We are now in position to complete the proof of the main theorems of the section.

\begin{proof}[Proof of Theorems~\ref{theo:graph-maps-ens} and~\ref{theo:graph-maps-ens-2}]
Let $\sigma$ be a one-edge nonseparating free splitting of $F_N$. Then $\sigma$ belongs to a finite set $\{\sigma_1,\dots,\sigma_N\}$ which determines a vertex of $K$. Since $f_K$ is surjective, by definition of $f_K$, it follows that $\sigma$ belongs to the image of $f$, showing surjectivity of $f$ on vertices. 
  
We now know that $f$ is an injective graph map of $\ens$ which is bijective on vertices, and we are left showing that its inverse is also a graph map, in other words that two nonadjacent vertices $v_1,v_2\in\ens$ cannot be sent to two adjacent vertices by $f$.
  
So assume that $f(v_1)$ and $f(v_2)$ are adjacent, and let us prove that $v_1$ and $v_2$ are adjacent. There exists a collection $\{f(v_1),f(v_2),v'_3,\dots,v'_N\}$ of vertices of $\ens$ which determines an $N$-rose. As $f_K$ is bijective, there exist vertices $w_1,\dots,w_N$ of $\ens$ which refine to an $N$-rose, with $v'_i=f(w_i)$ for every $i\in\{1,\dots,N\}$. As $f$ is injective, both $v_1$ and $v_2$ belong to $\{w_1,\dots,w_N\}$. In particular $v_1$ and $v_2$ are adjacent, as claimed.
\end{proof}

\subsection{More about rank $3$}

In the rank $3$ case, we will need to know a little more about the combinatorial properties of $\ens$ in order to prove our subgroup rigidity results later on. The following statement is specific to rank $3$.

\begin{lemma}\label{lemma:clique-rank-3}
Let $N=3$ and let $\Sigma$ be a clique of size 4 in $\ens$. Then the quotient graph of the splitting given by $\Sigma$ is either a cage with 4 edges or a theta graph with a loop attached at one of its vertices. In particular, complementary components of $\Sigma$ are simply connected.
\end{lemma}

\begin{figure}[h]\centering 
%% Creator: Inkscape inkscape 0.92.3, www.inkscape.org
%% PDF/EPS/PS + LaTeX output extension by Johan Engelen, 2010
%% Accompanies image file 'blow-ups.pdf' (pdf, eps, ps)
%%
%% To include the image in your LaTeX document, write
%%   \input{<filename>.pdf_tex}
%%  instead of
%%   \includegraphics{<filename>.pdf}
%% To scale the image, write
%%   \def\svgwidth{<desired width>}
%%   \input{<filename>.pdf_tex}
%%  instead of
%%   \includegraphics[width=<desired width>]{<filename>.pdf}
%%
%% Images with a different path to the parent latex file can
%% be accessed with the `import' package (which may need to be
%% installed) using
%%   \usepackage{import}
%% in the preamble, and then including the image with
%%   \import{<path to file>}{<filename>.pdf_tex}
%% Alternatively, one can specify
%%   \graphicspath{{<path to file>/}}
%% 
%% For more information, please see info/svg-inkscape on CTAN:
%%   http://tug.ctan.org/tex-archive/info/svg-inkscape
%%
\begingroup%
  \makeatletter%
  \providecommand\color[2][]{%
    \errmessage{(Inkscape) Color is used for the text in Inkscape, but the package 'color.sty' is not loaded}%
    \renewcommand\color[2][]{}%
  }%
  \providecommand\transparent[1]{%
    \errmessage{(Inkscape) Transparency is used (non-zero) for the text in Inkscape, but the package 'transparent.sty' is not loaded}%
    \renewcommand\transparent[1]{}%
  }%
  \providecommand\rotatebox[2]{#2}%
  \newcommand*\fsize{\dimexpr\f@size pt\relax}%
  \newcommand*\lineheight[1]{\fontsize{\fsize}{#1\fsize}\selectfont}%
  \ifx\svgwidth\undefined%
    \setlength{\unitlength}{417.30222206bp}%
    \ifx\svgscale\undefined%
      \relax%
    \else%
      \setlength{\unitlength}{\unitlength * \real{\svgscale}}%
    \fi%
  \else%
    \setlength{\unitlength}{\svgwidth}%
  \fi%
  \global\let\svgwidth\undefined%
  \global\let\svgscale\undefined%
  \makeatother%
  \begin{picture}(1,0.11450147)%
    \lineheight{1}%
    \setlength\tabcolsep{0pt}%
    \put(0,0){\includegraphics[width=\unitlength,page=1]{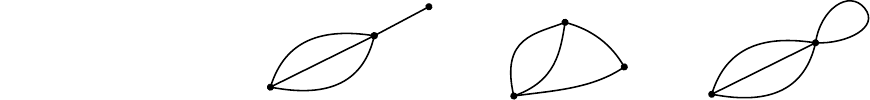}}%
    \put(0.50239242,0.10548711){\color[rgb]{0,0,0}\makebox(0,0)[lt]{\lineheight{1.25}\smash{\begin{tabular}[t]{l}$\mathbb{Z}$\end{tabular}}}}%
    \put(0.73218474,0.03616425){\color[rgb]{0,0,0}\makebox(0,0)[lt]{\lineheight{1.25}\smash{\begin{tabular}[t]{l}$\mathbb{Z}$\end{tabular}}}}%
    \put(0,0){\includegraphics[width=\unitlength,page=2]{blow-ups.pdf}}%
    \put(0.13780568,0.09521708){\color[rgb]{0,0,0}\makebox(0,0)[lt]{\lineheight{1.25}\smash{\begin{tabular}[t]{l}$\mathbb{Z}$\end{tabular}}}}%
  \end{picture}%
\endgroup%
 \caption{The possible one-edge blow-ups of a theta graph with one $\mathbb{Z}$-vertex when $N=3$.}
\label{fig:blow-ups}
\end{figure}

\begin{proof}
Let $\sigma_1, \sigma_2, \sigma_3$, and $\sigma_4$ be the four one-edge splittings in $\Sigma$, and let $\Sigma_i$ be the splitting given by the first $i$ elements of $\Sigma$. As elements of $\Sigma$ are pairwise rose compatible, the quotient graph of $\Sigma_2$ is a rose with 2 petals (and a copy of $\mathbb{Z}$ at its vertex when viewed as a graph of groups). As $\Sigma_3$ is a one-edge blow-up of $\Sigma_2$ and the elements of $\Sigma_3$ are pairwise rose compatible, the quotient graph of $\Sigma_3$ is either a 3-petal rose or a theta graph with $\mathbb{Z}$ at one of its vertices. In the first case, $\Sigma=\Sigma_4$ is a one-edge blow-up of the 3-rose, so is one of the `fat' graphs described in Section~\ref{s:surjectivity}. The only such graphs are the cage with 4 edges and the theta graph with a loop attached at one of its vertices. In the second case, $\Sigma_3$ is a theta graph with $\mathbb{Z}$ at one of its vertices, and has three possible one-edge blow-ups given in Figure~\ref{fig:blow-ups}. The first has a separating edge and the second contains a pair of edges that separate. As $\Sigma$ is a clique in $\ens$ this leaves the third case, which is a theta graph with a loop attached at one of its vertices. Both possibilities for $\Sigma$ have trivial vertex stabilizers.
\end{proof}

The following statement will also be useful.

\begin{lemma}\label{lemma:complex-rank-3}
Let $N=3$. Let $f$ be an injective graph map of $\ens$. Assume that no two compatible splittings are sent to a pair of splittings that intersect exactly once by $f$. 

Then $f$ extends to a graph map of $\ns$ (necessarily preserving the decoration $\calc_r$ given by the edges in $\ens$).
\end{lemma}

\begin{proof}
As we already know that $f$ gives an injective graph map of $\ens$, we only need to show that any two circle compatible one-edge nonseparating free splittings are sent to compatible splittings by $f$. Such splittings $S$ and $S'$ contain a clique $\Sigma$ of size four in their common link in $\ens$ (see, for example, the proof of \cite[Theorem~3.4]{HW}). The image of this clique under $f$ gives a splitting $f(\Sigma)$ compatible with both $f(S)$ and $f(S')$. Since $f(\Sigma)$ is also a clique of size $4$ in $\ens$, it follows from Lemma~\ref{lemma:clique-rank-3} that every complimentary component of $f(\Sigma)$ is simply connected, so that any two one-edge splittings that are both compatible with $f(\Sigma)$ are either compatible or determine spheres that intersect exactly once. As we have assumed that no two compatible splittings are sent to splittings that intersect exactly once by $f$, it follows that $f(S)$ and $f(S')$ are compatible, as desired. 
\end{proof}

\section{Commuting normal subgroups and invariant splittings}\label{sec:commuting-normal}

\emph{Stabilizers of one-edge nonseparating free splittings in $\ia$ contain two commuting normal free groups coming from the group of twists. The goal of this section is to use actions on free factor complexes to establish partial converses to this. In particular, we explain how commuting infinite normal subgroups can lead to a group having invariant splittings: this is the contents of Proposition~\ref{prop:splitting-stabilized}.}

\subsection{Commuting normal subgroups and actions on hyperbolic spaces}

We start with a general statement asserting that commuting normal subgroups give restrictions on possible actions on hyperbolic spaces. This is a variation of \cite[Proposition~4.2]{HW}.

\begin{lemma}\label{lemma:action-on-hyperbolic}
Let $X$ be a Gromov hyperbolic space, and let $H$ be a group acting by isometries on $X$. Assume that $H$ contains two normal subgroups $K_1$ and $K_2$ that centralize each other.

If $H$ has unbounded orbits in $X$ then there exists $i\in\{1,2\}$ such that $K_i$ has a finite orbit in $\partial_\infty X$. 
\end{lemma}

\begin{proof}
Gromov's classification of actions on hyperbolic spaces (see e.g.\ \cite[Proposition~3.1]{CCMT}) tells us that each $K_i$ either contains a loxodromic element, has a finite orbit in $\partial_\infty X$, or has bounded orbits in $X$. If $\Phi\in K_1$ acts loxodromically on $X$, then the pair $\{\Phi^{-\infty},\Phi^{+\infty}\}$ is $K_2$-invariant in  $\partial_\infty X$. Likewise, if $K_2$ contains a loxodromic element then $K_1$ has a finite orbit in  $\partial_\infty X$. This leaves the case where both $K_1$ and $K_2$ have bounded orbits in $X$. As $H$ has unbounded orbits, and as each $K_i$ is normal in $H$ and has bounded orbits in $X$, it follows from \cite[Lemma~4.3]{HW} that each $K_i$ fixes a point in $\partial_\infty X$ (in fact each $K_i$ fixes the entire limit set of $H$ in $\partial_\infty X$, which is nonempty). 
\end{proof}

\subsection{Commuting normal subgroups and invariant factors}

\begin{prop}\label{prop:commuting-normal-factor}
Let $H\subseteq\ia$ be a noncyclic subgroup. Assume that $H$ contains two normal infinite subgroups $K_1$ and $K_2$ that centralize each other.

Then there exists a proper free factor of $F_N$ whose conjugacy class is $H$-invariant. 
\end{prop}

\begin{proof}
Assume for a contradiction that $H$ does not preserve the conjugacy class of any proper free factor.  
By Proposition~\ref{prop:unbounded-orbits-in-ff}, the group $H$ acts on the free factor graph $\FF(=\FF(F_N,\emptyset))$ with unbounded orbits. Lemma~\ref{lemma:action-on-hyperbolic} then implies that either $K_1$ or $K_2$ fixes a point in $\partial_\infty\FF$. But stabilizers of points in $\partial_\infty\FF$ are virtually cyclic: indeed, if $K\subseteq\Out(F_N)$ fixes a point in $\partial_\infty\FF$, then Lemma~\ref{lemma:boundary-tree} shows that $K$ has a finite-index subgroup that fixes the projective class $[T]$ of an arational $F_N$-tree; by \cite{Rey}, either $T$ is a free $F_N$-action, in which case \cite{KL} implies that the $\Out(F_N)$-stabilizer of $[T]$ is virtually cyclic, or else $T$ is dual to an arational foliation on a surface with one boundary component, and the $\Out(F_N)$-stabilizer of $[T]$ is contained in the corresponding mapping class group, whence virtually cyclic. So either $K_1$ or $K_2$ (say $K_1$) is cyclic (being virtually cyclic and contained in the torsion-free group $\ia$), and in fact generated by a fully irreducible automorphism. Then $H$, which is contained in the normalizer of $K_1$ in $\ia$, is cyclic and generated by a fully irreducible automorphism, a contradiction. 
\end{proof}

\subsection{Commuting normal subgroups and invariant splittings}

Given $\Phi\in\Out(F_N)$, a splitting $U$ of $F_N$ is $\Phi$-invariant if for every lift $\phi$ of $\Phi$ in $\Aut(F_N)$, there exists a (unique) isometry $I_\phi$ of $U$ such that for every $x\in U$ and every $g\in F_N$, one has $I_\phi(gx)=\phi(g)I_\phi(x)$. The following proposition, which develops on \cite{GH}, will be of crucial importance in the present paper.

\begin{prop}\label{prop:splitting-stabilized}
Let $H\subseteq\ia$ be a subgroup. Assume that $H$ contains two normal infinite subgroups $K_1$ and $K_2$ that centralize each other. Assume in addition that $H$ contains a direct product of $2N-4$ nonabelian free groups. 

Then there exists a nontrivial $H$-invariant splitting $U$ of $F_N$ such that $H$ acts trivially on the quotient graph $U/F_N$. Furthermore, the vertex set $V(U)$ of $U$ has an $F_N$-invariant partition $V(U)=V_1\dunion V_2$ with the following properties: 
\begin{enumerate}
\item there exists $i\in\{1,2\}$ such that for every vertex $v\in V_1$, every outer automorphism in $K_i$ has a representative in $\Aut(F_N)$ which restricts to the identity on $G_v$, and
\item the collection of all conjugacy classes of stabilizers of vertices in $V_2$ is a free factor system of $F_N$.
\end{enumerate}
\end{prop}

\begin{remark}
We warn the reader that we do not claim anything about edge stabilizers of the splitting $U$, in particular $U$ need not be a free splitting or even a cyclic splitting of $F_N$.  
\end{remark}

In our proof of Proposition~\ref{prop:splitting-stabilized}, we will make use of the following theorem of Guirardel and the second named author (given by applying \cite[Theorem~6.12]{GH} with $H$ equal to the elementwise stabilizer of $\mathcal{C}$).

\begin{theo}[{\cite[Theorem~6.12]{GH}}]\label{theo:canonical-splitting}
There exists an $\Out(F_N)$-equivariant map which assigns to every collection $\mathcal{C}$ of free splittings of $F_N$ whose elementwise $\Out(F_N)$-stabilizer is infinite, a nontrivial splitting $U_\calc$ of $F_N$ whose vertex set $V(U_\calc)$ has an $F_N$-invariant partition $V(U_\calc)=V_1\dunion V_2$ with the following properties: 
\begin{enumerate}
\item For every vertex $v\in V_1$, the following hold:
\begin{enumerate}
\item either there exists an edge with trivial stabilizer incident on $v$, or the vertex stabilizer $G_v$ has a nontrivial Grushko decomposition relative to its incident edge stabilizers,
\item every outer automorphism in the elementwise stabilizer of the collection $\calc$ in $\ia$ has a representative in $\Aut(F_N)$ which restricts to the identity on $G_v$.
\end{enumerate}
\item The collection of all conjugacy classes of stabilizers of vertices in $V_2$ is a free factor system of $F_N$.
\end{enumerate}
\end{theo}

\begin{remark}
The idea behind the proof of Theorem~\ref{theo:canonical-splitting} is the following: denoting by $\tilde{K}$ the preimage of $K$ in $\Aut(F_N)$, every tree in $\calc$ can be viewed as a splitting of $\tilde{K}$. There is a $\tilde{K}$-deformation space $\mathcal{D}_\calc$ naturally associated to $\calc$, and the tree $U_\calc$ is then constructed as the tree of cylinders of this deformation space. 
\end{remark}

\begin{proof}[Proof of Proposition~\ref{prop:splitting-stabilized}]
If $H$ fixes a free splitting $U$ of $F_N$, we are done by adding the midpoint of every edge of $U$ in the vertex set and letting $V_1$ be the set of these extra valence two vertices and $V_2$ be the set of all original vertices in $U$. From now on, we will assume that $H$ does not fix any free splitting.

Let $\calf$ be a maximal $H$-invariant free factor system. As $H$ does not preserve any free splitting, $\calf$ is nonsporadic. Let $\FF:=\FF(F_N,\calf)$ be the corresponding free factor graph, which is hyperbolic. By Proposition~\ref{prop:unbounded-orbits-in-ff}, the group $H$ acts on $\FF$ with unbounded orbits. Therefore, Lemma~\ref{lemma:action-on-hyperbolic} ensures that there exists $i\in\{1,2\}$ such that $K_i$ virtually fixes a point in $\partial_\infty\FF$. 

If $K_i$ contains a loxodromic element $\Phi$, then there is a unique maximal $K_i$-invariant finite set in $\partial_\infty\FF$, which has to be contained in $\{\Phi^{-\infty},\Phi^{+\infty}\}$. By normality, this set is $H$-invariant, so Lemma~\ref{lemma:boundary-tree} implies that $H$ virtually fixes the homothety class of an arational $(F_N,\calf)$-tree $T$. As $H$ contains a direct product of $2N-4$ nonabelian free groups, this yields a contradiction to \cite[Lemma~6.5]{HW}.

We can therefore assume that $K_i$ contains no loxodromic element. Then $K_i$ virtually fixes the isometry class of an arational tree (rel.\ $\calf$), as opposed to only fixing its homothety class \cite[Proposition~6.3]{GH1}. By \cite[Proposition~6.10]{GH1}, it follows that $K_i$ fixes a free splitting of $F_N$ (notice that the assumption from \cite[Proposition~6.10]{GH1} that $K_i$ has \emph{finite fix type} is automatically satisfied in view of \cite[Lemma~6.7]{GH}). Let $\calc$ be the (nonempty) collection of all $K_i$-invariant free splittings, and let $U_\calc$ be the splitting provided by Theorem~\ref{theo:canonical-splitting}, which is nontrivial because $K_i$ is infinite. By normality, this splitting is $H$-invariant.

We are left showing that $H$ acts trivially on the quotient graph $U_\calc/F_N$. Let us first prove that $H$ preserves the orbit of every vertex $v$ in $V_1$. Let $\cala_v$ be the collection of all conjugacy classes of stabilizers of vertices of $U_\calc$ outside of the orbit of $v$. As $H$ does not preserve any free splitting, all edges in $U$ have nontrivial stabilizer. Therefore, Theorem~\ref{theo:canonical-splitting} ensures that $G_v$ is freely decomposable relative to the incident edge stabilizers. Blowing up $U_\calc$ at $G_v$ into a free splitting shows that $F_N$ is freely decomposable relative to $\cala_v$. Let $\calf_v$ be the free factor system of $F_N$ determined by a Grushko decomposition of $F_N$ relative to $\cala_v$. Since $H$ preserves $U_\calc$, the $H$-orbit of $\calf_v$ is finite. As $H\subseteq\ia$ we deduce that $H$ preserves $\calf_v$. But $\calf_v$ entirely determines $v$, as $G_v$ is (up to conjugation) the only vertex stabilizer of $U_\calc$ which is not elliptic in $\calf_v$. This shows that $H$ preserves the orbit of $v$ in $U_\calc$. 

So far we have proved that $H$ acts on $U_\calc/F_N$ fixing all vertices in $V_1/F_N$, and therefore permuting the vertices in $V_2/F_N$. As the stabilizers of vertices in $V_2$ form a free factor system of $F_N$ and $H\subseteq\ia$, it follows that $H$ also fixes every vertex in $V_2/F_N$. This completes our proof.     
\end{proof}

\section{Twist-rich subgroups}\label{sec:twist-rich}

\emph{In this section we introduce a slightly weaker notion for twist-rich subgroups of $\Out(F_N)$ than the one used in \cite{HW}. In contrast to the previous section, we show that twist-rich subgroups do not preserve any free factors and do not contain commuting infinite normal subgroups. Crucially, we also show that stabilizers of corank one free factors in twist-rich subgroups do not preserve any other free factors, in particular they do not preserve any free splittings other than the obvious one defined by the factor.}
\\

Recall that an element $w$ of a finitely generated free group $F$ is \emph{simple} if it is contained in some proper free factor of $F$, and \emph{nonsimple} otherwise. We now introduce our definition of \emph{twist-rich} subgroups of $\Out(F_N)$, which are the subgroups to which our methods apply.

\begin{de}[Twist-rich subgroups]\label{de:twist-rich}
Let $N\ge 4$. A subgroup $\Gamma\subseteq\ia$ is \emph{twist-rich} if for every free splitting $S$ of $F_N$ such that $S/F_N$ is a rose with nonabelian vertex stabilizer $G_v$, and every half-edge $e$ adjacent to $v$, the intersection of $\Gamma$ with the group of twists about $e$ is nonabelian and viewed as a subgroup of $G_v$, it is not contained in any proper free factor of $G_v$. 
\end{de}

We allow a `rose with a single petal' for $S/F_N$, in which case $S$ is simply a  one-edge nonseparating free splitting. Notice that the last assumption that the group of twists about $e$ is not contained in any proper free factor of $G_v$, can equivalently be replaced by the intersection of $\Gamma$ with the group of twists about $e$ contains a nonsimple element   -- see the proof of \cite[Proposition~7.3(4)]{HW}. Notice that finite-index subgroups of twist-rich subgroups of $\Out(F_N)$ are again twist-rich.  

\begin{remark}
This definition of twist-richness is different from, and \emph{a priori} weaker than (though we do not have explicit new examples) the one given in \cite[Definition~7.1]{HW}: Hypothesis~$(H_1)$ from there has been removed. The condition about existence of nonsimple elements in the group of twists is slightly weaker than the condition in Hypothesis~$(H_2)$ from \cite[Definition~7.1]{HW}, and we only need to look at roses rather than arbitrary free splittings.
\end{remark}

\subsection{Stabilizers of free factors and free splittings in twist-rich groups} 

Roughly speaking, twist-rich subgroups mix up free factors and free splittings as much as  possible. We make this idea precise via some lemmas below, and use this to deduce some algebraic consequences for twist-rich groups.

\begin{lemma}\label{lemma:twist-rich-factor}
Let $N\ge 3$, and let $\Gamma$ be a twist-rich subgroup of $\ia$. Then $\Gamma$ does not preserve the conjugacy class of any proper free factor of $F_N$.
\end{lemma}

\begin{proof}
Let $A$ be a proper free factor of $F_N$, and let $x_1, \ldots, x_N$ be a basis of $F_N$ extending a basis $x_1, \ldots, x_k$ of $A$. Let $B$ be the corank one factor generated by $x_2, \ldots, x_N$, so that there is a one-edge nonseparating free splitting determined by $B$ where $x_1$ is a possible choice of stable letter. As $\Gamma$ is twist-rich, there exists a filling element $b \in B$ and a right twist $\Phi \in \Gamma$ represented by an automorphism $\phi$ sending $x_1$ to $x_1b$ and fixing all other $x_i$. In particular, $b$ is represented by a word in $\{x_2, \ldots, x_N\}$ that is not supported on any smaller subset. It follows that $x_1b$ is not supported on any subset of the basis $x_1, \ldots, x_N$, so that $\phi(A) \neq A$. In fact, $\phi(A)$ is not even conjugate to $A$. If $A$ is cyclic, this is because both $x_1$ and $x_1b$ are cyclically reduced when viewed as words, and they have different lengths. If $A$ is not cyclic, then the intersection $\phi(A) \cap A$ is nontrivial, containing $\langle x_2, \ldots, x_k \rangle$. As free factors are malnormal, this implies that $A$ and $\phi(A)$ are not conjugate. 
\end{proof}

Furthermore, stabilizers of one-edge nonseparating free splittings (or equivalently, stabilizers of conjugacy classes of corank one factors) in twist-rich subgroups do not preserve any other free splittings. 

\begin{lemma}\label{lemma:twist-rich-unique-splitting}
Let $N\ge 3$, let $\Gamma$ be a twist-rich subgroup of $\ia$, and let $S$ be a one-edge nonseparating free splitting of $F_N$. Then $S$ is the only one-edge nonseparating free splitting of $F_N$ which is invariant by $\Stab_\Gamma(S)$.
\end{lemma}

\begin{proof}
Let $S'$ be a free splitting of $F_N$ which is invariant by $\Stab_\Gamma(S)$. Let $A\subseteq F_N$ be a corank one free factor such that $S$ is the Bass--Serre tree of the HNN decomposition $F_N=A\ast$. Let $t$ be a stable letter for this HNN decomposition. As $\Gamma$ is twist-rich, there exists a filling element $z \in A$ such that the twist $D$ mapping $t$ to $tz$ and fixing $A$ belongs to $\Gamma$. Let $w$ be the unique smallest (positive) root of $z$ and let $U=(A \ast \langle twt^{-1} \rangle )\ast_ {\langle w \rangle}$ be the cyclic splitting obtained by fully folding the edge $e$ given by this stable letter with its translate $we$. As $D(S')=S'$ and $U$ is an attracting tree for the twist, it follows from \cite[Lemma~2.7]{HW} that $S'$ is compatible with $U$. But as $z$, and therefore $w$, is nonsimple in $A$, the only free splitting which is compatible with $U$ is $S$ (see \cite[Proposition~5.1]{HW2}). Therefore $S'=S$, which completes our proof. 
\end{proof}

It is also crucial to note that the stabilizer $\Stab_\G(S)$ of a one-edge nonseparating free splitting in $\G$ (and therefore $\G$ itself) is of maximal product rank in $\Out(F_N)$:

\begin{lemma}\label{lemma:product-rank-stabilizer-twist-rich}
Let $N\ge 3$, let $\Gamma$ be a twist-rich subgroup of $\ia$, and let $S$ be a one-edge nonseparating free splitting of $F_N$.  Then $\Stab_\Gamma(S)$ contains a direct product of $2N-4$ nonabelian free groups.
\end{lemma}

\begin{proof}
Every nonseparating free splitting $S$ has a refinement $S'$ whose quotient graph is a rose with $N-2$ petals. Each vertex stabilizer $G_v$ in this refinement is rank 2 and the group of twists of $S'$ is a direct product of $2N-4$ copies of $G_v$ (see Section~\ref{sec:twists-background}). As $\G$ is twist-rich, $\Stab_\Gamma(S')$ contains a direct product of $2N-4$ free groups appearing as a subgroup of these twists, and as we are in $\ia$, we have $\Stab_\Gamma(S') \subset \Stab_\G(S)$.
\end{proof}

We attain the following algebraic consequences for twist-rich subgroups:

\begin{cor}\label{cor:normal-centralized-in-twist-rich}
Let $N\ge 3$, and let $\Gamma$ be a twist-rich subgroup of $\ia$. Then $\Gamma$ does not contain two infinite normal subgroups that centralize each other. 
\end{cor}

\begin{proof}
If $\Gamma$ contains two infinite normal subgroups that centralize each other, then as $\Gamma$ is noncyclic, Proposition~\ref{prop:commuting-normal-factor} implies that $\Gamma$ fixes the conjugacy class of a free factor. This cannot happen by Lemma~\ref{lemma:twist-rich-factor}.
\end{proof}

Corollary~\ref{cor:normal-centralized-in-twist-rich} implies in particular that every twist-rich subgroup $\Gamma$ of $\Out(F_N)$ is centerless (as otherwise $\Gamma$ and its center would be two infinite normal subgroups of $\Gamma$ that centralize each other). In fact we have the following, where we recall that given a group $G$ and a subgroup $H\subseteq G$, the \emph{weak centralizer} of $H$ in $G$ is the subgroup of $g$ made of all elements $g\in G$ such that for every $h\in H$, the element $g$ commutes with some nontrivial power of $h$. 

\begin{cor}\label{cor:twist-rich-centerless}
Let $N\ge 3$, and let $H\subseteq\ia$ be a twist-rich subgroup. Then the weak centralizer of $H$ in $\Out(F_N)$ is trivial. 
\end{cor}

\begin{proof}
Let $\Gamma$ be a twist-rich subgroup of $\ia$, and let $\Phi\in\Out(F_N)$ be contained in the weak centralizer of $\Gamma$. We will prove that $\Phi$ preserves every one-edge nonseparating free splitting of $F_N$, from which it follows that $\Phi=\id$.

Let $S$ be such a splitting. By Lemma~\ref{lemma:twist-rich-unique-splitting}, the splitting $S$ is the only one-edge nonseparating free splitting which is invariant by $\Stab_\Gamma(S)$. As we are working in $\ia$, the splitting $S$ is also the only one-edge nonseparating free splitting which is invariant by a power of every element in  $\Stab_\Gamma(S)$. As $\Phi$ weakly centralizes $\Stab_\Gamma(S)$, it permutes the set of such splittings, and by uniqueness $S$ is $\Phi$-invariant, which concludes our proof. 
\end{proof}

\subsection{More on stabilizers of corank one free factors} 

The goal of the present section is to prove the following statement, which will complete our analysis of stabilizers of corank one factors in twist-rich groups in rank $N\ge 4$. It will be used in Section~\ref{sec:injectivity}.

\begin{prop}\label{prop:unique-invariant-factor}
Let $N\ge 4$, let $\Gamma$ be a twist-rich subgroup of $\ia$, and let $A\subseteq F_N$ be a corank one free factor. If  $B\subseteq F_N$ is a proper free factor such that $[B]$ is preserved by $\Stab_\Gamma([A])$, then $B$ is conjugate to $A$.
\end{prop}

We first require the following lemma, which states that the twist-rich property is inherited under some restriction maps.

\begin{lemma}\label{lemma:twist-rich-restriction} 
Let $N\ge 4$, let $\Gamma$ be a twist-rich subgroup of $\ia$, let $A$ be a corank one free factor of $F_N$, and let $\Gamma_A$ be the image of $\Stab_\Gamma([A])$ in $\Out(A)$ (under the natural map).

Then $\Gamma_A$ is twist-rich in $\Out(A)$. 
\end{lemma}

\begin{proof}
Let $S_A$ be a free splitting of $A$ such that $S_A /A$ is a rose with nonabelian vertex group. We can add another petal to $S_A$ to obtain a free splitting $S$ of $F_N$ whose quotient graph of groups is also a rose with the same vertex stabilizer, and whose minimal $A$-invariant subtree is $S_A$. The conclusion then follows by applying twist-richness of $\Gamma$ to the splitting $S$.
\end{proof}

This allows us to deal with the case where $B$ is conjugate into $A$.

\begin{lemma}\label{lemma:twist-rich-restriction-factor}
Let $N\ge 4$, let $\Gamma$ be a twist-rich subgroup of $\ia$, let $A$ be a corank one free factor of $F_N$, and let $\Gamma_A$ be the image of $\Stab_\Gamma([A])$ in $\Out(A)$ (under the natural map).

Then $\Gamma_A$ does not preserve the conjugacy class of any proper free factor of $A$.
\end{lemma}

\begin{proof}
By Lemma~\ref{lemma:twist-rich-restriction}, the group $\Gamma_A$ is twist-rich. The conclusion thus follows from Lemma~\ref{lemma:twist-rich-factor}  (notice that we are using $N\ge 4$ to ensure that $A$ has rank at least $3$).
\end{proof}

We finally need a short technical lemma before completing the proof of Proposition~\ref{prop:unique-invariant-factor}.

\begin{lemma}\label{lemma:whitehead}
Let $A$ be a corank one free factor of $F_N$, and let $\langle t\rangle$ be a rank one complementary free factor to $A$. Let $w\in A$ be an element which is nonsimple in $A$. 

Then every simple element of $F_N$ contained in $A\ast\langle twt^{-1}\rangle$ is either conjugate into $A$ or conjugate to $twt^{-1}$, in particular it is conjugate into $A$ in $F_N$.  
\end{lemma}

\begin{proof}
Our proof relies on Whitehead's algorithm \cite{Whi, Sta}. Associated to each free basis $\calb$ of $F_N$ and element $g \in F_N$, there is a \emph{Whitehead graph} $\Wh_\calb(g)$ whose vertex set is $\calb \cup \calb^{-1}$ and edge set consists of the pairs $(a,b)$ such that $ab^{-1}$ is a subword of the cyclic word given by the cyclic reduction of $g$. If $\Wh_\calb(g)$ is connected without a cut vertex, then $g$ is nonsimple, and for every nonsimple $g$ there is some basis in which $\Wh_\calb(g)$ has this property \cite[Theorem~2.4 and Algorithm~2.5]{Sta}.

Let $g\in A\ast\langle twt^{-1}\rangle$ be an element which is neither conjugate into $A$ nor conjugate to $twt^{-1}$; we aim to prove that $g$ is nonsimple. As $w$ is nonsimple, there exists a free basis $\calb_A$ of $A$ in which the Whitehead graph  $\Wh_{\calb_A}(w)$ is connected without a cut point, and $w$ is cyclically reduced. Write $w=x_1\ldots x_n$ in the basis $\calb_A$. For all $a \in A$, the set $\calb:= \calb_A \cup \{at\}$ is a free basis of $F_N$. We choose $a \in A$ such that the word length $|g|_\mathcal{B}$ is minimal, and set $t'=at$.  As $g$ is not conjugate into $A$ and $g\in A\ast\langle twt^{-1}\rangle=A\ast\langle t'w{t'}^{-1}\rangle$, the Whitehead graph $\Wh_\calb(g)$ contains a copy of $\Wh_{\calb_A}(w)$, except that the edge between $x_n$ and $x_1^{-1}$ might be removed. As $w$ is cyclically reduced in $\mathcal{B}_A$ there are two distinct edges from $t'$ in $\Wh_\calb(g)$ to $x_1^{-1}$ and to $x_n$. 
It follows that $\Wh_\calb(g)$ is connected without a cut vertex as long as there are edges from two distinct vertices in $\calb_A \cup \calb_A^{-1}$ to $t'^{-1}$; this is the case as otherwise, there would exist $a_0 \in \calb_A \cup \calb_A^{-1}$ which always precedes $t'$ and follows $t'^{-1}$, and either $g$ would be contained in a conjugate of $\langle t'w{t'}^{-1}\rangle$ or be shorter in the basis  $\calb_A\cup\{a_0t'\}$ than in $\calb$. Both of these lead to contradictions. This shows that $\Wh_\calb(g)$ is connected without a cut vertex, and therefore $g$ is nonsimple.
\end{proof}

\begin{proof}[Proof of Proposition~\ref{prop:unique-invariant-factor}] 
Suppose that $[B]$ is preserved by $H:=\Stab_\Gamma([A])$, and assume for a contradiction that $[B]\neq[A]$. Let $S$ be the Bass--Serre tree of the free splitting $F_N=A\ast$. Let $S_B \subseteq S$ be the minimal subtree of $B$ with respect to its action on $S$ and let $H_B$ be the image of $H$ in $\Out(B)$. As both $[B]$ and $S$ are $H$-invariant, the tree $S_B$ is $H_B$-invariant.

The vertex stabilizers of $S_B$ are given by the intersections $A^h \cap B$ with $h \in F_N$. This gives a finite collection of conjugacy classes of subfactors of $A$ and $B$, which are invariant by a finite-index subgroup of $H$ (in fact by $H$ because we are working in $\ia$). By Lemma~\ref{lemma:twist-rich-restriction-factor}, the group $H$ does not preserve a proper free factor of $A$ and so these intersections are trivial. Hence the action of $B$ on $S_B$ is free and its stabilizer in $\mathrm{IA}(B,\mathbb{Z}/3\mathbb{Z})$ is trivial. As $H_B$ is contained in this stabilizer, it follows that the $H$-action on $B$ is trivial (i.e.\ $H_B$ is trivial).

If $t$ is chosen to be a stable letter for $S$ (so that $\langle A, t \rangle = F_N$), then as $\Gamma$ is twist-rich, there exists a nonsimple, root-free $w \in A$ such that the twist $\phi$ sending $t \mapsto tw$ and fixing $A$ has some power belonging to $H$. Any group fixed by a representative of $\phi$, or indeed fixed by any power of $\phi$, is conjugate into $A\ast\langle twt^{-1}\rangle$ (this can be seen either by directly computing fixed subgroups of representatives via train-tracks \cite{CT, BH}, or by looking at the limiting tree $T_\phi$ of $\phi$ in the boundary of outer space, whose vertex groups are the conjugates of $A\ast\langle twt^{-1}\rangle$ \cite{CL}).

Therefore, Lemma~\ref{lemma:whitehead} implies that every element $b \in B$ (which is simple because $B$ is a free factor) is conjugate into $A$. This contradicts the fact (proven in the previous paragraph) that the action of $B$ on $S_B$ is free.
\end{proof}

\section{From algebra to graph maps and rigidity}\label{sec:algebra}

\emph{In this section we prove our main theorem in the case $N \geq 4$, utilizing the analysis of stabilizers of corank one factors in twist-rich subgroups from the previous two sections.}
\\

The first goal of the present section is to prove the following statement.

\begin{theo}\label{theo:algebra}
Let $N\ge 4$, let $\Gamma$ be a twist-rich subgroup of $\ia$, and let $f:\Gamma\to\ia$ be an injective homomorphism.
\begin{enumerate}
\item For every one-edge nonseparating free splitting $S$ of $F_N$, there exists a unique one-edge nonseparating free splitting $f_\ast(S)$ of $F_N$ such that $f(\Stab_\Gamma(S))\subseteq\Stab_{\ia}(f_\ast(S))$.  
\item The map $S\mapsto f_\ast(S)$ is injective.
\item If $S_1$ and $S_2$ are rose compatible, then $f_\ast(S_1)$ and $f_\ast(S_2)$ are rose compatible.
\end{enumerate}
\end{theo}

Our main theorem (Theorem~\ref{theo:intro-main}) will then easily follow from Theorem~\ref{theo:algebra} and the strong rigidity of $\ens$, as explained in Section~\ref{sec:conclusion} below. The most involved part in the proof of Theorem~\ref{theo:algebra} is the existence of the splitting $f_\ast(S)$: this will be done in Section~\ref{sec:existence} through an algebraic characterization of stabilizers of one-edge nonseparating free splittings. Uniqueness of $f_\ast(S)$ will be proved in Section~\ref{sec:uniqueness}, and injectivity of the map $S\mapsto f_\ast(S)$ will be established in Section~\ref{sec:injectivity}, by taking advantage of maximal direct products of free groups. The last point follows from work of Bridson and the third named author \cite{BW}. All of this is summarized in Section~\ref{sec:conclusion}. In Section~\ref{sec:complements}, we will explain how to derive the corollaries announced in the introduction about abstract commensurators and the co-Hopfian property.

\subsection{Free splitting stabilizers are sent into free splitting stabilizers.}\label{sec:existence}

In this section, we will give a proof of the following proposition.

\begin{prop}\label{prop:stab-split-to-stab-split}
Let $N\ge 4$. Let $\Gamma\subseteq\ia$ be a twist-rich subgroup, and let $f:\Gamma\to\ia$ be an injective homomorphism. Then for every one-edge nonseparating free splitting $S$ of $F_N$, there exists a one-edge nonseparating free splitting $S'$ of $F_N$ such that $f(\Stab_\Gamma(S))\subseteq\Stab_{\ia}(S')$.
\end{prop}

To prove Proposition~\ref{prop:stab-split-to-stab-split}, we will first establish in Proposition~\ref{prop:stab-splitting-property} a few algebraic properties of stabilizers of one-edge nonseparating free splittings in twist-rich subgroups of $\ia$. We will then show (Proposition~\ref{prop:stab-splitting-characterization}) that conversely, any subgroup of $\Out(F_N)$ that satisfies these properties has to virtually stabilize a one-edge nonseparating free splitting. This will be enough to conclude.

\begin{prop}\label{prop:stab-splitting-property}
Let $N\ge 4$. Let $\Gamma\subseteq\ia$ be a twist-rich subgroup, let $S$ be a one-edge nonseparating free splitting of $F_N$, and let $H:=\Stab_\Gamma(S)$. Then
\begin{enumerate}
\item[$(P_1)$] $H$ contains a direct product of $2N-4$ nonabelian free groups.
\item[$(P_2)$] $H$ contains a direct product $K_1\times K_2$ of two nonabelian free groups with $K_1$ and $K_2$ both normal in $H$.
\item[$(P_3)$] Whenever $H$ contains two infinite normal subgroups $K_1$ and $K_2$ that centralize each other, then both $K_1$ and $K_2$ are free. 
\end{enumerate}
\end{prop}

\begin{remark}\label{rk:centerless}
Property~$(P_3)$ implies in particular that $H$ is centerless: indeed, letting $K_1=H$ (which is not free in view of Property~$(P_1)$), and $K_2$ be the center of $H$, Property~$(P_3)$ implies that $K_2$ is finite, whence trivial as $H\subseteq\ia$. 
\end{remark}

\begin{proof}
Property~$(P_1)$ follows from the definition of twist-rich subgroups applied to an $(N-2)$-petal rose refinement of $S$: the group of twists of such a refinement is a direct product of $2N-4$ nonabelian free groups, and therefore $H$ will also contain a direct product of $2N-4$ nonabelian free groups. 

We now prove Property~$(P_2)$. By \cite{Lev}  (as recalled in Section~\ref{sec:twists-background}), the group of twists in $\Out(F_N)$ of the splitting $S$ is isomorphic to a direct product $\mathcal{T}_1\times \mathcal{T}_2$ of two nonabelian free groups (isomorphic to $F_{N-1}$). As $\Gamma\subseteq\ia$, the intersections of $\calt_1$ and $\calt_2$ with $\Gamma$ are both normal in $H$, and these intersections are nonabelian because $\Gamma$ is twist-rich.

Finally, we prove Property~$(P_3)$. Let $K_1$ and $K_2$ be two infinite normal subgroups of $H$ that centralize each other. Let $K_1^A$ and $K_2^A$ be their respective images in $\mathrm{IA}(A,\mathbb{Z}/3\mathbb{Z})$, and let $H^A$ be the image of $H$, which is twist-rich by Lemma~\ref{lemma:twist-rich-restriction}.
The subgroups $K_1^A$ and $K_2^A$ are both normal in the twist-rich subgroup $H^A$, and they centralize each other. It thus follows from Corollary~\ref{cor:normal-centralized-in-twist-rich} that either $K_1^A$ or $K_2^A$, say $K_1^A$, is trivial. This implies that $K_1$ is contained in the group of twists of $S$. As $H$ has non-abelian intersection with both the group of left and right twists, normality of $K_1$ implies that $K_1$ intersects either the group of left twists or the group of right twists (say right) in a nonabelian subgroup $T_1$. In particular $T_1$ is normal in the group of right twists in $H$. Since the group $T$ of right twists in $H$, viewed as a subgroup of $A$, is not contained in any proper free factor of $A$, and $T_1$ is normal in $T$, it follows that $T_1$ is not contained in any proper free factor of $A$ (otherwise the smallest free factor of $A$ containing $T_1$ would be $T$-invariant).

We claim that $K_2^A$ is also trivial. Indeed, assume towards a contradiction that $K_2^A$ is infinite. We fix a choice of a stable letter $t$ of the splitting $F_N=A\ast$ and use it to identify $T_1$ with a subgroup of $A$. For every $w\in T_1$, every  $\Phi\in K_2$ commutes with the right twist $t\mapsto tw$. For  $\Phi\in K_2$, letting $a,b\in A$ be such that $\Phi$ has a representative $\phi$ in $\Aut(F_N)$ that preserves $A$ and sends $t$ to $a tb$, we deduce that $a t b \phi(w)=a t w b$ for every $w\in T_1$. Thus every  element of $K_2^A$ preserves the conjugacy class of every element $w\in T_1$. Let $\Fix(K_2^A)$ be the set of all conjugacy classes of elements of $A$ which are fixed by $K_2^A$. In particular $\Fix(K_2^A)$ contains all conjugacy classes of elements of $T_1$, and therefore $A$ is freely indecomposable relative to $\Fix(K_2^A)$. We can therefore apply \cite[Theorem~9.5]{GL-jsj}, and get a $\zmax$ JSJ tree $J$ of $A$ relative to $\Fix(K_2^A)$ which is $H^A$-invariant. This tree $J$ is nontrivial: otherwise, as $K_2^A$ is infinite, it would be reduced to a QH vertex relative to $\Fix(K_2^A)$ (as defined in \cite[Definition~5.13]{GL-jsj}). This would imply that all conjugacy classes in $\Fix(K_2^A)$ correspond to boundary subgroups under some identification of $A$ with the fundamental group of a compact surface with boundary. This is impossible because $T_1$ is nonabelian. Since $J$ is a $\zmax$ splitting of $A$, all edge stabilizers of $J$ are simple in $A$ (see e.g.\ \cite{She,Swa}). Therefore $H^A$ fixes the conjugacy class of a proper free factor of $A$ -- namely, the smallest proper free factor of $A$ that contains one of the edge stabilizers. As $H^A$ is twist-rich, this contradicts Lemma~\ref{lemma:twist-rich-factor}. This contradiction shows that $K_2^A$ is trivial. 

Therefore $K_2$ is also contained in the group of twists of $S$. Now, as $K_1$ and $K_2$ are two nonabelian subgroups of $A\times A$ such that each of them has nonabelian intersection with at least one of the factors, and centralize each other, it follows that up to exchanging their roles $K_1$ is contained in the group of right twists and $K_2$ is contained in the group of left twists. In particular they are free.    
\end{proof}

The following statement is a converse to Proposition~\ref{prop:stab-splitting-property}.

\begin{prop}\label{prop:stab-splitting-characterization}
Let $N\ge 4$. Let $H\subseteq\ia$ be a subgroup. Assume that $H$ satisfies the following three conditions.
\begin{enumerate}
\item[$(P_1)$] $H$ contains a direct product of $2N-4$ nonabelian free groups.
\item[$(P_2)$] $H$ contains a subgroup which is isomorphic to a direct product $K_1\times K_2$ of two nonabelian free groups, with $K_1$ and $K_2$ both normal in $H$.
\item[$(P_3)$] Whenever $H$ contains two infinite normal subgroups $K'_1$ and $K'_2$ that centralize each other, then both $K'_1$ and $K'_2$ are free. 
\end{enumerate}
Then $H$ fixes a one-edge nonseparating free splitting of $F_N$. 
\end{prop}

\begin{proof}
Let $U$ be a $H$-invariant splitting of $F_N$ provided by Proposition~\ref{prop:splitting-stabilized} (using Properties~$(P_1)$ and~$(P_2)$); its vertex set comes with a partition $V(U)=V_1\dunion V_2$, keeping the notation from that proposition. 
 
If $U$ is a free splitting, then as $H\subseteq\ia$ it fixes every one-edge collapse of $U$ (Theorem~\ref{theo:ia}). Since $H$ contains a direct product of $2N-4$ nonabelian free groups by assumption (Hypothesis~$(P_1)$), it follows from \cite[Theorem~6.1]{HW} that $H$ cannot fix a one-edge separating free splitting of $F_N$. So $H$ fixes a one-edge nonseparating free splitting of $F_N$ and we are done.

We now assume that $U$ is not a free splitting, and actually does not contain any edge with trivial stabilizer (otherwise such an edge determines a $H$-invariant free splitting). We aim for a contradiction. Recall that $H$ acts trivially on the quotient graph $U/F_N$. 

The intersection of $H$ with the group of twists in $U$ is trivial: we are assuming that there are no edges in $U$ with trivial stabilizer, and $H$ cannot intersect the twist subgroup coming from cyclic edges as such elements would be central in $H$ (see \cite[Lemma~5.3]{CL2}), contradicting Hypothesis~$(P_3)$ (see Remark~\ref{rk:centerless}). Higher rank edge groups have trivial center and do not contribute to the group of twists. It follows  from \cite[Proposition~2.2]{Lev} that the map $$H\to\prod_{v\in V(U)/F_N}\Out(G_v),$$ given by the action on the vertex groups, is injective. (There is a further technical point here, as we also need to rule out \emph{bitwists} in the sense of \cite[Section~2.4]{Lev}. Bitwists of a splitting that are not twists come from non-inner automorphisms of an edge group $G_e$ induced by its normalizers in both adjacent vertex groups.  They do not appear in the stabilizer of $U$ as every edge group is a free factor, whence malnormal, in one of the incident vertex groups.)  

For every $i\in\{1,2\}$, we let $K_i'$ be the subgroup of $H$ made of all automorphisms whose image in $\prod_{v\in V_i/F_N}\Out(G_v)$ is trivial. Then $K_1'$ and $K_2'$ are normal in $H$, and they centralize each other.  
It follows from the first property of $U$ given by Proposition~\ref{prop:splitting-stabilized} that some $K_i \cap H$ is an infinite subgroup of $K_1'$. We claim that $K_2'$ is not a free group (in fact that its product rank is at least $2$). This claim will contradict Hypothesis~$(P_3)$, and conclude our proof.

Recall that $K_2'$ is the kernel of the map $$H\to\prod_{v\in V_2/F_N}\Out(G_v).$$ By the second property of $U$ given in Proposition~\ref{prop:splitting-stabilized}, the collection of all conjugacy classes of groups $G_v$ with $v\in V_2$ is a free factor system of $F_N$. This either consists of a single corank 1 factor, or at least two factors whose ranks sum to at most $N$, and we have: $$\rk_{\pro}\left(\prod_{v\in V_2/F_N}\Out(G_v)\right)\le 2N-6.$$ Since $\rk_{\pro}(H)=2N-4$, using the bound on the product rank for extensions (Lemma~\ref{lemma:rank-product-extension}), we deduce that $\rk_{\pro}(K_2')\ge 2$, which proves our claim.
\end{proof}

We are now in position to complete our proof of Proposition~\ref{prop:stab-split-to-stab-split}.

\begin{proof}[Proof of Proposition~\ref{prop:stab-split-to-stab-split}]
Let $S$ be a one-edge nonseparating free splitting of $F_N$. By Proposition~\ref{prop:stab-splitting-property}, the group $\Stab_\Gamma(S)$ satisfies Properties~$(P_1)$--$(P_3)$. As $f$ is an injective homomorphism, it follows that $f(\Stab_\Gamma(S))$ also satisfies Properties~$(P_1)$--$(P_3)$. Proposition~\ref{prop:stab-splitting-characterization} thus ensures that $f(\Stab_\Gamma(S))$ fixes a one-edge nonseparating free splitting of $F_N$, as desired.
\end{proof}

\subsection{Uniqueness of $S'$}\label{sec:uniqueness}

We now prove that the free splitting $S'$ provided by Proposition~\ref{prop:stab-split-to-stab-split} is unique, namely.

\begin{prop}\label{prop:uniqueness}
Let $N\ge 4$. Let $\Gamma\subseteq\ia$ be a twist-rich subgroup, and let $f:\Gamma\to\ia$ be an injective homomorphism. Let $S$ be a one-edge nonseparating free splitting of $F_N$. 

Then there exists at most one one-edge nonseparating free splitting $S'$ of $F_N$ such that $f(\Stab_\Gamma(S))\subseteq\Stab_{\ia}(S')$.
\end{prop}

\begin{proof}
As $\Gamma$ is twist-rich, we have $\rk_{\pro}(\Stab_\Gamma(S))=2N-4$ (Lemma~\ref{lemma:product-rank-stabilizer-twist-rich}). Therefore $\rk_{\pro}(f(\Stab_\Gamma(S)))=2N-4$. 

Assume towards a contradiction that $f(\Stab_\Gamma(S))$ is contained in the $\Out(F_N)$-stabilizers of two distinct splittings $S'_1$ and $S'_2$. The third  assertion from Theorem~\ref{theo:product-rank-out-aut} implies that $S'_1$ and $S'_2$ are rose compatible.

There remains to exclude the possibility that $f(\Stab_\Gamma(S))$ be contained in the stabilizer of a two-petal rose $R$. The group of twists of the splitting $R$ in $\Out(F_N)$ is a direct product $A_1\times A_2\times A_3\times A_4$ of four nonabelian free groups (each isomorphic to $F_{N-2}$). We claim that the intersection of $f(\Stab_\Gamma(S))$ with two of the groups $A_i$ is trivial. Otherwise, $f(\Stab_\Gamma(S))$ contains a subgroup isomorphic to a direct product of three infinite  groups $B_1\times B_2\times B_3$, with each $B_i$ normal in $f(\Stab_\Gamma(S))$. As $f$ is injective, it follows that $\Stab_\Gamma(S)$ contains three infinite commuting normal subgroups. Taking $K_1 = f^{-1}(B_1)$ and $K_2=f^{-1}(B_2 \times B_3)$  contradicts Property~$(P_3)$ from Proposition~\ref{prop:stab-splitting-property}, and proves our claim. 

Now considering the action on the vertex group of the two-petal rose,  there is a map $$\Stab_{\ia}(R)\to \out(F_{N-2}),$$ whose kernel $K$ is contained in $A_1\times A_2\times A_3\times A_4$. By the above, the product rank of $K\cap f(\Stab_\Gamma(S))$ is at most $2$. It follows from Lemma~\ref{lemma:rank-product-extension} that the product rank of $f(\Stab_\Gamma(S))$ is at most $3$ if $N=4$, and at most $2N-6$ if $N\ge 5$, a contradiction.   
\end{proof}

\subsection{Injectivity}\label{sec:injectivity}

The above two sections yield a map $S\mapsto S'$ at the level of the vertex set of $\ens$. We now aim to prove that this map is injective. We show the following.

\begin{prop}\label{prop:injectivity}
Let $N\ge 4$. Let $\Gamma$ be a twist-rich subgroup of $\ia$, and let $f:\Gamma\to\ia$ be an injective homomorphism. Let $S_1$ and $S_2$ be two distinct one-edge nonseparating free splittings.

Then $\langle f(\Stab_\Gamma(S_1)), f(\Stab_\Gamma(S_2))\rangle$ does not fix any one-edge nonseparating free splitting.
\end{prop}

\begin{proof}
Let $H=\langle\Stab_\Gamma(S_1),\Stab_\Gamma(S_2)\rangle$, so that $f(H)=\langle f(\Stab_\Gamma(S_1)),f(\Stab_\Gamma(S_2))\rangle$. Let us assume that $f(H)$ fixes a one-edge nonseparating free splitting $S$ and work towards a contradiction. 

As $\Gamma$ is twist-rich, Proposition~\ref{prop:unique-invariant-factor} ensures that the only conjugacy class of a proper free factor fixed by each group $\Stab_\Gamma(S_i)$  is the one given by the single $F_N$-orbit of vertices in $S_i$. As $S_1$ and $S_2$ are distinct, this implies that $H$ does not fix the conjugacy class of any proper free factor of $F_N$.  Proposition~\ref{prop:commuting-normal-factor} then implies that $H$ cannot contain two infinite normal subgroups that centralize each other. As $f$ is injective, the same is true of $f(H)$.

Let $K$ be the group of twists about $S$ in $\Out(F_N)$, which is isomorphic to a direct product of two nonabelian free groups. It follows from the above that the group $f(H)$ intersects at most one of the two factors of $K$ nontrivially.  Hence $$\rk_{\pro}(K\cap f(H))\le 1.$$ The action on the vertex group of $S$ gives a homomorphism $$ f(H) \to \out(F_{N-1}), $$ 
whose kernel is $K\cap f(H)$. As the image has product rank at most $2N-6$, Lemma~\ref{lemma:rank-product-extension} implies that the product rank of $f(H)$ is at most $2N-5$. This is a contradiction because $\Gamma$ is twist-rich, so the product rank of $\Stab_{\Gamma}(S_1)$ is equal to $2N-4$ (Lemma~\ref{lemma:product-rank-stabilizer-twist-rich}).
\end{proof}

\subsection{Conclusion} \label{sec:conclusion}

\begin{proof}[Proof of Theorem~\ref{theo:algebra}]
The first point of Theorem~\ref{theo:algebra} was proved in Propositions~\ref{prop:stab-split-to-stab-split} (existence) and~\ref{prop:uniqueness} (uniqueness). Injectivity of the map $S\mapsto f_\ast(S)$ was established in Proposition~\ref{prop:injectivity}. As $\Gamma$ is twist-rich, the $\Gamma$-stabilizer of a free splitting $S$ such that $S/F_N$ is a two-petal rose contains a direct product of $2N-4$ nonabelian free groups. Therefore the fact that $f_\ast(S_1)$ and $f_\ast(S_2)$ are rose compatible whenever $S_1$ and $S_2$ are follows from the third assertion of Theorem~\ref{theo:product-rank-out-aut}.
\end{proof}

We are now in position to complete the proof of our main theorem.

\begin{theo}
Let $N\ge 4$. Then every twist-rich subgroup of $\Out(F_N)$ is rigid in $\Out(F_N)$.
\end{theo}

\begin{proof}
Let $f:\Gamma\to\Out(F_N)$ be an injective homomorphism. The group $\Out(F_N)$ acts on $\ens$ by graph automorphisms. We will check that the assumptions from Proposition~\ref{prop:blueprint} are satisfied. By Corollary~\ref{cor:strongly-rigid}, the graph $\ens$ is strongly rigid as an $\Out(F_N)$-graph. Let $\Gamma':=\Gamma\cap\ia\cap f^{-1}(\ia)$: this is a normal subgroup of $\Gamma$, and it is still twist-rich because it has finite index in $\Gamma$. The assumption from Proposition~\ref{prop:blueprint} is therefore exactly the content of Theorem~\ref{theo:algebra}. The conclusion thus follows from Proposition~\ref{prop:blueprint}.
\end{proof}

\subsection{Consequences: co-Hopf property and commensurations}\label{sec:complements}

We recall that a group $H$ is \emph{co-Hopfian} if every injective map $H\to H$ is an automorphism.

\begin{cor}\label{cor:co-hopf}
Let $N\ge 4$, and let $H\subseteq\Out(F_N)$ be a subgroup which is commensurable in $\Out(F_N)$ to a normal twist-rich subgroup of $\Out(F_N)$. Then $H$ is co-Hopfian.
\end{cor}

\begin{proof}
Let $G\unlhd\Out(F_N)$ be a normal twist-rich subgroup of $\Out(F_N)$ such that $H_0=H\cap G$ is finite-index in both $G$ and $H$. Let $f:H\to H$ be an injective homomorphism, and let $\iota:H\to\Out(F_N)$ be the inclusion map. Theorem~\ref{theo:intro-main} implies that $\iota\circ f$ is the conjugation by an element $g\in\Out(F_N)$. In particular $gHg^{-1}\subseteq H$. As $G$ is normal, we have $gGg^{-1}=G$, and $[G:gH_0g^{-1}]=[G:H_0]$. Therefore $gH_0g^{-1}=H_0$, showing that the conjugation by $g$ (which coincides with $f$) is an automorphism of $H_0$. We have the inclusion \[ H_0 \subseteq gHg^{-1}  \subseteq H \] Iterating the conjugation by $g$ gives a nested sequence of finite index subgroups \[H_0 \subseteq \cdots \subseteq g^{k+1}Hg^{-(k+1)} \subseteq g^{k}Hg^{-k} \subseteq \cdots \subseteq H\]
As the index of each $g^kHg^{-k}$ in $H$ is bounded above by the index of $H_0$ in $H$, the sequence stabilizes and there exists $K$ such that  $g^{k}Hg^{-k}=g^KHg^{-K}$ for all $k \geq K$. Let $H_1=g^KHg^{-K}$ be this finite index subgroup. Then $f$ is also an isomorphism restricted to $H_1$. If $h \in H$ then $g^K h g^{-K} \in H_1$. As $f^K$ is an automorphism of $H_1$ and $f^K$ is injective on $H$, this implies that $h \in H_1$, so that $H=H_1$ and $f$ is an automorphism of $H$. 
\end{proof}

We now turn to computing the automorphism group and the abstract commensurator of a twist-rich subgroup of $\Out(F_N)$, thus recovering the main result from \cite{HW} -- and extending it to an \emph{a priori} slightly larger class of subgroups. 

\begin{cor}
Let $N\ge 4$, and let $\Gamma$ be a twist-rich subgroup of $\Out(F_N)$. Then
\begin{enumerate}
\item The natural map $N_{\Out(F_N)}(\Gamma)\to\Aut(\Gamma)$ is an isomorphism.
\item The natural map $\Comm_{\Out(F_N)}(\Gamma)\to\Comm(\Gamma)$ is an isomorphism.
\end{enumerate}
\end{cor}

\begin{proof}
The injectivity of the first map follows from the fact that the centralizer of a twist-rich subgroup in $\Out(F_N)$ is trivial (Corollary~\ref{cor:twist-rich-centerless}). As finite index subgroups of $\G$ are also twist-rich, this also implies that the second map is injective. We concentrate on proving surjectivity in both cases.

For the first assertion, let $f:\Gamma\to\Gamma$ be an automorphism. Let $\iota:\Gamma\to\Out(F_N)$ be the inclusion. Theorem~\ref{theo:intro-main} implies that there exists $g\in\Out(F_N)$ such that $\iota\circ f$ is conjugation by $g$. In particular $g\Gamma g^{-1}=\Gamma$, so $g\in N_{\Out(F_N)}(\Gamma)$. This shows that the map $N_{\Out(F_N)}(\Gamma)\to\Aut(\Gamma)$ is surjective.

For the second assertion, let $\Gamma_1$ and $\Gamma_2$ be two finite-index subgroups of $\Gamma$, and let $f:\Gamma_1\to\Gamma_2$ be an isomorphism. Let $\iota:\Gamma_2\to\Out(F_N)$ be the inclusion map. Theorem~\ref{theo:intro-main} implies that there exists $g\in\Out(F_N)$ such that $\iota\circ f$ is conjugation by $g$. In particular $g\Gamma_1 g^{-1}=\Gamma_2$, so $g\in\Comm_{\Out(F_N)}(\Gamma)$. This shows that the map $\Comm_{\Out(F_N)}(\Gamma)\to\Comm(\Gamma)$ is surjective. 
\end{proof}

\section{The case of $\Out(F_3)$}\label{sec:rank-3}

\emph{In this section, we prove our main theorem in rank $3$.} 
\\

We recall that an outer automorphism $\Phi$ of $F_N$ is a \emph{Nielsen automorphism} if there exists a free basis $\{x_1,\dots,x_N\}$ of $F_N$ such that $\Phi$ has a representative in $\Aut(F_N)$ that sends $x_1$ to $x_1x_2$ and leaves all other basis elements invariant.

\begin{theo}\label{theo:rank-3}
Let $\Gamma\subseteq\Out(F_3)$ be a subgroup such that every Nielsen automorphism has a power contained in $\Gamma$. Then $\Gamma$ is rigid in $\Out(F_3)$.
\end{theo}

Our proof of Theorem~\ref{theo:rank-3} will follow roughly the same steps as in higher rank, with a few modifications in the details of the arguments.

\subsection{Stabilizers of free splittings in $\out(F_3)$}

Recall that two one-edge nonseparating free splittings $S_1$ and $S_2$ are \emph{rose compatible} if they are compatible and have a common refinement $U$ such that $U/F_3$ is a two-petal rose, and are \emph{circle compatible} if they have a common refinement $U$ such that $U/F_3$ is a circle with two vertices. 

\begin{lemma}\label{lemma:common-stab-rank-3} 
Let $\Gamma$ be a subgroup of $\mathrm{IA}_3(\mathbb{Z}/3\mathbb{Z})$ such that every Nielsen automorphism has a power contained in $\Gamma$. Let $S_1$ and $S_2$ be two distinct one-edge nonseparating free splittings of $F_3$.

If $S_1$ and $S_2$ are rose compatible, then $\Stab_{\Gamma}(S_1)\cap\Stab_{\Gamma}(S_2)$ is isomorphic to $\mathbb{Z}^3$.

If $S_1$ and $S_2$ are circle compatible, then $\Stab_{\Gamma}(S_1)\cap\Stab_{\Gamma}(S_2)$ is isomorphic to $\mathbb{Z}^2$.
\end{lemma}

\begin{proof}  
This is a consequence of work of Levitt \cite{Lev} described in Section~\ref{sec:twists-background}. For convenience, we describe these groups explicitly. If $S_1$ and $S_2$ are rose compatible, then the intersection $\Stab_{\Gamma}(S_1)\cap\Stab_{\Gamma}(S_2)$ is the stabilizer of a two-petal rose in $\G$. If $\{a,b,c\}$ is a generating set of $F_3$ such $\langle a \rangle$ is the vertex group and $b$ and $c$ represent the petals, then every twist in this splitting has a unique representative of the form $a \mapsto a$, $b \mapsto a^{k}ba^l$, $c \mapsto ca^m$. Hence $\G$ intersects this group in a finite-index subgroup. Furthermore, automorphisms in $\iat$ that stabilize this splitting are contained in the group of twists. A similar proof works in the circle compatible case, where the intersection is the stabilizer in $
\G$ of a two-edge loop  (here, one can choose a basis such that $a$ and $b$ stabilize the two vertices and every twist has a unique representative of the form $a \mapsto a$, $b \mapsto b$, $c \mapsto a^lcb^k$).\end{proof}

In particular, the above lemma implies that the joint stabilizer of two distinct, compatible, one-edge nonseparating free splittings in $\iat$ is abelian.

\begin{proposition} \label{prop:fix-splitting-unique-rank-3} 
Let $K=K_1 \times K_2$ be a direct product of nonabelian free groups in $\iat$. Then $K$ fixes a unique one-edge nonseparating free splitting.
\end{proposition}

\begin{proof}
As $K$ is a maximal direct product of nonabelian free groups in $\iat$, the existence of such an invariant splitting follows from the second assertion of Theorem~\ref{theo:product-rank-out-aut}. If $K$ fixed two such splittings, they would be rose compatible by the third assertion of Theorem~\ref{theo:product-rank-out-aut}. This contradicts Lemma~\ref{lemma:common-stab-rank-3}.
\end{proof}

\subsection{Free splitting stabilizers are sent into free splitting stabilizers.}

The algebraic characterization of stabilizers is simpler in the case $N=3$. We use the following condition:

\begin{prop}\label{prop:property-stabilizer-rank-3}
Let $\Gamma\subseteq\iat$ be a subgroup such that every Nielsen automorphism of $F_3$ has a power contained in $\Gamma$. Let $S$ be a one-edge nonseparating free splitting of $F_3$.

Then $\Stab_\Gamma(S)$ contains two normal nonabelian free subgroups that centralize each other.
\end{prop}

\begin{proof}
Consider the intersections of $\Gamma$ with the group of left twists and right twists about the splitting $S$. These are normal, commuting subgroups of $\Stab_\Gamma(S)$. In $\Out(F_3)$, if a vertex group of the splitting is the free factor generated by $a$ and $b$ and $t$ is a stable letter, then the group of left twists is a free group of rank 2 generated by a pair of Nielsen automorphisms fixing $a$ and $b$ and mapping $t$ to $at$ and $bt$ respectively. Hence if $\G$ contains a power of every Nielsen automorphism then the intersection of $\G$ with the group of left twists is nonabelian. Similarly, the same holds for the group of right twists. 
\end{proof}

\begin{prop}\label{prop:characterization-stabilizer-rank-3}
Let $H\subseteq\iat$ be a subgroup which contains two normal nonabelian free groups that centralize each other.

Then $H$ fixes a unique one-edge nonseparating free splitting of $F_3$. 
\end{prop}
 
\begin{proof}
Let $K=K_1 \times K_2$ be a normal direct product of nonabelian free groups in $H$. By Proposition~\ref{prop:fix-splitting-unique-rank-3}, $K$ fixes a unique one-edge nonseparating free splitting of $F_3$. As this splitting is unique, it is also invariant under the normalizer of $K$ in $\Out(F_3)$, which contains $H$. As $K \subseteq H$ and $K$ fixes a unique one-edge nonseparating free splitting, the same is true of $H$.
\end{proof}

As a consequence of Propositions~\ref{prop:property-stabilizer-rank-3} and~\ref{prop:characterization-stabilizer-rank-3}, we deduce the following fact.

\begin{cor}\label{cor:stab-split-to-stab-split-rank-3}
Let $\Gamma\subseteq\iat$ be a subgroup such that every Nielsen automorphism has a power contained in $\Gamma$, and let $f:\Gamma\to\iat$ be an injective  homomorphism. Then for every one-edge nonseparating free splitting $S$ of $F_3$, there exists a unique one-edge nonseparating free splitting $S'$ of $F_3$ such that $f(\Stab_\Gamma(S))\subseteq\Stab_{\iat}(S')$.
\end{cor}

\begin{proof}
By Proposition~\ref{prop:property-stabilizer-rank-3}, the group $\Stab_\Gamma(S)$ contains two normal nonabelian free subgroups that centralize each other. Therefore, so does $f(\Stab_\Gamma(S))$. The conclusion then follows from Proposition~\ref{prop:characterization-stabilizer-rank-3}.
\end{proof}

\subsection{Injectivity}

Corollary~\ref{cor:stab-split-to-stab-split-rank-3}  
yields a well-defined map $S\mapsto S'$ on the vertex set of $\ens$. After a preliminary lemma, we now prove that this map is injective.

\begin{lemma}\label{lemma:direct-product-stab-rank-3}
Let $S$ be a one-edge nonseparating free splitting of $F_3$, and let $H\subseteq\Stab_{\iat}(S)$ be a subgroup that contains a direct product of two nonabelian free groups.

Then $H$ contains two normal, nonabelian free subgroups that centralize each other. 
\end{lemma}

\begin{proof}
Let $A\subseteq F_3$ be a rank two free factor such that $S$ is the Bass--Serre tree of the decomposition $F_3=A\ast$, and let $\{a,b,c\}$ be a free basis of $F_3$ such that $A=\langle a,b\rangle$. The group of twists about $S$ in $\Out(F_3)$ splits as a direct product $K_1\times K_2$ of two nonabelian free groups (both isomorphic to $A$). It is enough to check that $H$ intersects both $K_1$ and $K_2$ in nonabelian free groups. 

Let $\Stab^0(S)$ be the finite-index subgroup of the $\Out(F_3)$-stabilizer of $S$ made of all automorphisms acting trivially on the quotient graph $S/F_3$. Then $\Stab^0(S)$ splits as a semidirect product \[ \Stab^0(S) \cong K_1 \rtimes \Aut(F_2), \] where $K_1$ is as before and $K_2$ projects to the inner automorphisms in $\Aut(F_2)$ (see, e.g., \cite[Lemma~1.28]{GS}). As $H$ contains a direct product of two nonabelian free groups and $\aut(F_2)$ does not (Theorem~\ref{theo:product-rank-out-aut}), the group $H$ intersects $K_1$ in a nonabelian free group. Using the analogous semidirect product decomposition with $K_2$ shows that $H$ also intersects $K_2$ is a nonabelian free group. 
\end{proof}

\begin{prop}\label{prop:injectivity-rank-3}
Let $\Gamma\subseteq\iat$ be a subgroup such that every Nielsen automorphism has a power contained in $\Gamma$. Let $f:\Gamma\to\iat$ be an injective map. Let $S_1$ and $S_2$ be two one-edge nonseparating free splittings of $F_3$.

If $\langle f(\Stab_\Gamma(S_1)),f(\Stab_\Gamma(S_2))\rangle$ fixes a one-edge nonseparating free splitting, then $S_1=S_2$. 
\end{prop}

\begin{proof}
Let $H=\langle\Stab_\Gamma(S_1),\Stab_\Gamma(S_2)\rangle$, so that $f(H)=\langle f(\Stab_\Gamma(S_1)),f(\Stab_\Gamma(S_2))\rangle$. Suppose that $f(H)$ fixes a one-edge nonseparating free splitting. As each $\Stab_{\Gamma}(S_i)$ contains a direct product of two nonabelian free groups, so does $H$ and therefore so does $f(H)$. Therefore by Lemma~\ref{lemma:direct-product-stab-rank-3}, the group $f(H)$ contains two normal, nonabelian free subgroups that centralize each other. The same is then true of $H$, so that $H$ fixes a one-edge nonseparating free splitting $S$ by Proposition~\ref{prop:characterization-stabilizer-rank-3}. This splitting $S$ is fixed by both $\Stab_\Gamma(S_1)$ and $\Stab_\Gamma(S_2)$, however as these subgroups contain a direct product of two nonabelian free groups, they fix at most one such splitting (Proposition~\ref{prop:fix-splitting-unique-rank-3}). Hence $S=S_1=S_2$.
\end{proof}

\subsection{Compatibility}

The work so far is enough to show that an injective homomorphism $f:\G \to \iat$ induces an injective map on the vertex set of $\ens$. We now move on to looking at what happens to the edges (i.e. compatibility).

\begin{lemma}\label{lemma:rank3-1}
Let $\Gamma$ be a subgroup of $\ia$ such that every Nielsen automorphism has a power contained in $\Gamma$. Let $S$ and $S'$ be two one-edge nonseparating free splittings of $F_3$. 

Then $S$ and $S'$ are rose compatible if and only if $\Stab_\Gamma(S)\cap\Stab_\Gamma(S')$ is isomorphic to $\mathbb{Z}^3$. 
\end{lemma}

\begin{proof}
If $S$ and $S'$ are rose compatible, then we have already seen in Lemma~\ref{lemma:common-stab-rank-3} that $\Stab_\Gamma(S)\cap\Stab_\Gamma(S')$ is isomorphic to $\mathbb{Z}^3$. 

If $S$ and $S'$ are circle compatible, then by Lemma~\ref{lemma:common-stab-rank-3} their common stabilizer in $\Gamma$ is isomorphic to $\mathbb{Z}^2$ so we are done. 

We now assume that $S$ and $S'$ are not compatible. By Lemma 5.4 of \cite{BW}, any subgroup of $\iat$ that fixes two incompatible free splittings, fixes a free splitting with at least two orbits of vertices. However, a case-by-case analysis shows that the size of the stabilizer of such a splitting in $\iat$ does not contain a $\mathbb{Z}^3$.
%We will assume that the intersection $\Stab_\Gamma(S)\cap\Stab_\Gamma(S')$ is  isomorphic to $\mathbb{Z}^3$ and aim for a contradiction.  As the group $\Stab_\Gamma(S)\cap\Stab_\Gamma(S')$ is contained in $\mathrm{IA}_3(\mathbb{Z}/3\mathbb{Z})$, it fixes every sphere that arises on a surgery path from $S$ to $S'$\scom{not terribly important, but we've been avoiding surgeries so far...}. Therefore we can find a one-edge free splitting $S_1$, which is distinct from $S$, compatible with $S'$, and invariant by $\Stab_\Gamma(S)\cap\Stab_\Gamma(S')$.  Let $U$ be the common refinement of $S_1$ and $S'$. The stabilizer of a one-edge separating free splitting in $\Out(F_3)$ does not contain $\mathbb{Z}^3$ (it is isomorphic to $\aut(F_2)$), so $S_1$ is also nonseparating and Lemma~\ref{lemma:common-stab-rank-3} implies that $U/F_3$ is a two-petal rose and $\Stab_{\Out(F_3)}(U)$ is commensurable to $\Stab_\Gamma(S)\cap\Stab_\Gamma(S')$. By \cite[Lemma~7.6]{HW}, the two collapses $S_1$ and $S'$ of $U$ are the only one-edge free splittings that are virtually fixed by $\Stab_{\Out(F_3)}(U)$, which contradicts that $S$ is also virtually fixed by this group.    
\end{proof}

For the extra conditions needed for graph rigidity in the $N=3$ case, we also require the following lemma.

\begin{lemma}\label{lemma:rank3-3}
Let $S$ and $S'$ be two one-edge nonseparating free splittings which intersect exactly once. Then $\Stab_{\Out(F_3)}(S)\cap\Stab_{\Out(F_3)}(S')$ does not contain any subgroup isomorphic to $\mathbb{Z}^2$.
\end{lemma}

\begin{proof}
In that case $\Stab_{\Out(F_3)}(S)\cap\Stab_{\Out(F_3)}(S')$ fixes the boundary splitting of $S$ and $S'$. However, by using a case-by-analysis of the possible boundary splittings given in Figure~\ref{fig:boundary-spheres} together with the result of Levitt given in Proposition~\ref{prop:Levitt}, one can show that the stabilizer of any boundary splitting in $\out(F_3)$ is virtually cyclic (possibly finite). 
\end{proof}

\begin{figure}[ht]  \centering 
%% Creator: Inkscape inkscape 0.92.3, www.inkscape.org
%% PDF/EPS/PS + LaTeX output extension by Johan Engelen, 2010
%% Accompanies image file 'arc_example.pdf' (pdf, eps, ps)
%%
%% To include the image in your LaTeX document, write
%%   \input{<filename>.pdf_tex}
%%  instead of
%%   \includegraphics{<filename>.pdf}
%% To scale the image, write
%%   \def\svgwidth{<desired width>}
%%   \input{<filename>.pdf_tex}
%%  instead of
%%   \includegraphics[width=<desired width>]{<filename>.pdf}
%%
%% Images with a different path to the parent latex file can
%% be accessed with the `import' package (which may need to be
%% installed) using
%%   \usepackage{import}
%% in the preamble, and then including the image with
%%   \import{<path to file>}{<filename>.pdf_tex}
%% Alternatively, one can specify
%%   \graphicspath{{<path to file>/}}
%% 
%% For more information, please see info/svg-inkscape on CTAN:
%%   http://tug.ctan.org/tex-archive/info/svg-inkscape
%%
\begingroup%
  \makeatletter%
  \providecommand\color[2][]{%
    \errmessage{(Inkscape) Color is used for the text in Inkscape, but the package 'color.sty' is not loaded}%
    \renewcommand\color[2][]{}%
  }%
  \providecommand\transparent[1]{%
    \errmessage{(Inkscape) Transparency is used (non-zero) for the text in Inkscape, but the package 'transparent.sty' is not loaded}%
    \renewcommand\transparent[1]{}%
  }%
  \providecommand\rotatebox[2]{#2}%
  \newcommand*\fsize{\dimexpr\f@size pt\relax}%
  \newcommand*\lineheight[1]{\fontsize{\fsize}{#1\fsize}\selectfont}%
  \ifx\svgwidth\undefined%
    \setlength{\unitlength}{298.93463615bp}%
    \ifx\svgscale\undefined%
      \relax%
    \else%
      \setlength{\unitlength}{\unitlength * \real{\svgscale}}%
    \fi%
  \else%
    \setlength{\unitlength}{\svgwidth}%
  \fi%
  \global\let\svgwidth\undefined%
  \global\let\svgscale\undefined%
  \makeatother%
  \begin{picture}(1,0.27873046)%
    \lineheight{1}%
    \setlength\tabcolsep{0pt}%
    \put(0,0){\includegraphics[width=\unitlength,page=1]{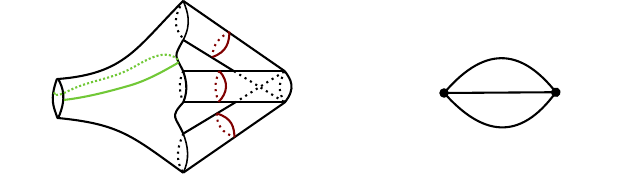}}%
    \put(0.90752685,0.12880792){\color[rgb]{0,0,0}\makebox(0,0)[lt]{\lineheight{1.25}\smash{\begin{tabular}[t]{l}$\mathbb{Z}$\end{tabular}}}}%
    \put(0.66011923,0.12694487){\color[rgb]{0,0,0}\makebox(0,0)[lt]{\lineheight{1.25}\smash{\begin{tabular}[t]{l}$F_3$\end{tabular}}}}%
    \put(0.79368296,0.14084139){\color[rgb]{0,0,0}\makebox(0,0)[lt]{\lineheight{1.25}\smash{\begin{tabular}[t]{l}$\mathbb{Z}$\end{tabular}}}}%
    \put(0.80180999,0.04018279){\color[rgb]{0,0,0}\makebox(0,0)[lt]{\lineheight{1.25}\smash{\begin{tabular}[t]{l}$\mathbb{Z}$\end{tabular}}}}%
    \put(0.79248608,0.19453679){\color[rgb]{0,0,0}\makebox(0,0)[lt]{\lineheight{1.25}\smash{\begin{tabular}[t]{l}$\mathbb{Z}$\end{tabular}}}}%
    \put(-0.0010413,0.1170448){\color[rgb]{0,0,0}\makebox(0,0)[lt]{\lineheight{1.25}\smash{\begin{tabular}[t]{l}$X=$\end{tabular}}}}%
    \put(0.15195839,0.21572959){\color[rgb]{0,0,0}\makebox(0,0)[lt]{\lineheight{1.25}\smash{\begin{tabular}[t]{l}$S$\end{tabular}}}}%
  \end{picture}%
\endgroup%
 \caption{One obtains a space $X$ with fundamental group $F_3$ by taking a four-holed sphere $S$ and identifying three of the four boundary components (depicted as annuli above). The three Dehn twists about curves in these annuli fix every arc in $S$ based on the unused boundary component and generate a subgroup of $\Out(F_3)$ isomorphic to $\mathbb{Z}^2$.  Even though the arcs determine nonseparating free splittings, this does not contradict Lemma~\ref{lemma:rank3-3}, as crossing arcs in $S$ intersect at least twice. The splitting induced by the three curves is given on the right.} 
\label{f:arc_example} 
\end{figure}
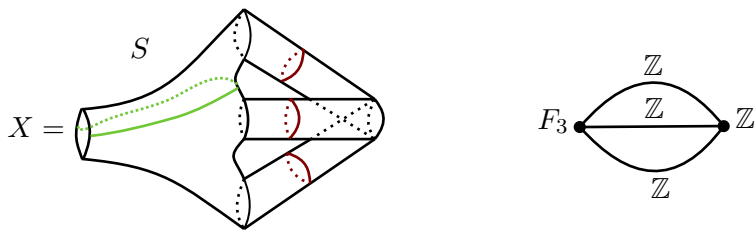

\begin{remark}
 We end this section with an important warning regarding Lemma~\ref{lemma:rank3-3}. If the common stabilizer of $S$ and $S'$ contains a copy of $\mathbb{Z}^2$, this is not enough to deduce that these splittings are compatible. Indeed, there are subgroups isomorphic to $\mathbb{Z}^2$ in $\out(F_3)$ that fix an infinite number of one-edge nonseparating free splittings. See Figure~\ref{f:arc_example}.    
\end{remark}

\subsection{End of the proof}

\begin{proof}[Proof of Theorem~\ref{theo:rank-3}]
Let $\calc_r=\{E_r\}$ be the decoration of $\ns$ given by the edges in $\ens$ (i.e.\ those representing rose compatibility). By Corollary~\ref{cor:strongly-rigid}, the decorated graph $(\ns,\calc_r)$ is a strongly rigid $\Out(F_3)$-graph.

Let $f:\Gamma\to\Out(F_3)$ be an injective map. Let \[ \Gamma':=\Gamma\cap\iat\cap f^{-1}(\iat).\] This is a finite-index, normal subgroup of $\Gamma$. In particular, $\G'$ still contains a power of every Nielsen automorphism. Corollary~\ref{cor:stab-split-to-stab-split-rank-3} gives a map  $\theta$ on the vertex set of $\ens$ such that for every $S \in \ens$ one has \[ f(\Stab_{\Gamma'}(S))\subseteq \Stab_{\Out(F_3)}(\theta(S)). \] Furthermore, the same corollary tells us that this choice of $\theta(S)$ is unique.  By Proposition~\ref{prop:injectivity-rank-3}, the map $\theta$ is also injective and by Lemma~\ref{lemma:rank3-1}, the map $\theta$ preserves rose compatibility (so is a graph map of $\ens$). 

If $S$ and $S'$ are compatible, then their common stabilizer in $\G'$ contains $\mathbb{Z}^2$ (by Lemma~\ref{lemma:common-stab-rank-3}), which means that $\theta(S)$ and $\theta(S')$ do not intersect exactly once (Lemma~\ref{lemma:rank3-3}). The extra work on graph rigidity in the case $N=3$ (Lemma~\ref{lemma:complex-rank-3}) tells us that $\theta$ extends to an injective graph map of $\ns$ preserving the decoration $\calc_r$. The conclusion therefore follows from the blueprint in  Proposition~\ref{prop:blueprint}.
\end{proof}

Corollary~\ref{cor:intro-co-hopf-3} follows from Theorem~\ref{theo:rank-3} in the same way as in higher rank.

\bibliographystyle{alpha}
\bibliography{hhw-bib-2}
~
\begin{flushleft}
Sebastian Hensel\\
Mathematisches Institut der Universität München\\
D-80333 München\\
\emph{e-mail: }\texttt{hensel@math.lmu.de}
\end{flushleft}
~
\begin{flushleft}
Camille Horbez\\ 
Laboratoire de Math\'ematique d'Orsay, Univ.\ Paris-Sud, CNRS, Universit\'e Paris-Saclay\\ 
F-91405 Orsay\\
\emph{e-mail: }\texttt{camille.horbez@universite-paris-saclay.fr}
\end{flushleft}
~
\begin{flushleft}
Richard D.\ Wade\\
Mathematical Institute, University of Oxford\\
Oxford OX2 6GG\\
\emph{e-mail: }\texttt{wade@maths.ox.ac.uk} 
\end{flushleft}

\end{document}